\documentclass[12pt]{amsart}
\usepackage{amsmath,amssymb, mathtools,a4wide} % basic maths packages
\usepackage{amscd} % for commutative diagrams (rectangular)
\usepackage{amsthm} % for theorem environments
\usepackage{dsfont} % for mathematical fonts like mathbb{R}
\usepackage{MnSymbol} % for matrix dots

\usepackage{float} % for fixing pictures where I want (with [H] in \begin{figure}[H]) but interferes with hyperref...

\usepackage[pdftex]{graphicx} % for including pictures with includegraphics
\usepackage{xcolor} % for colors like textcolor{red}
\usepackage[colorlinks,citecolor=blue,linkcolor=blue,urlcolor=blue]{hyperref} % for references in the text with colors%\usepackage[nottoc,numbib]{tocbibind} % pour afficher references dans la table des matieres MAIS fait disparaitre les sections en haut...

\usepackage{geometry} % for size, margins, ...
\usepackage[utf8]{inputenc} % to translate accents to intern latex code

\usepackage{tikz} % for pictures with tikz
\newtheorem{thm}{Theorem}
\newtheorem{conj}[thm]{Conjecture}
\newtheorem{lemma}[thm]{Lemma}
\newtheorem*{fakeprop}{Proposition}
\newtheorem*{fakethm}{Theorem}
\newtheorem{coro}[thm]{Corollary}
\newtheorem{definition}[thm]{Definition}
\newtheorem{prop}[thm]{Proposition}
\newtheorem{example}[thm]{Example}
\newtheorem*{Remark}{Remark}

\numberwithin{thm}{section}
\numberwithin{equation}{section}

\newcommand*{\R}{\mathbb{R}}
\newcommand*{\C}{\mathbb{C}}

\renewcommand*{\S}{\Sigma}
\newcommand*{\g}{\mathfrak{g}}
\newcommand*{\h}{\mathfrak{h}}

\newcommand*{\mf}{\mathfrak}
\newcommand*{\mc}{\mathcal}
\newcommand*{\bb}{\mathbb}
\newcommand*{\del}{\partial}

\newcommand*{\T}{\hat{\mc{T}}}

\DeclareMathOperator{\tr}{tr}
\DeclareMathOperator{\ad}{ad}
\DeclareMathOperator{\Sym}{Sym}
\DeclareMathOperator{\Symp}{Symp}
\DeclareMathOperator{\Hilb}{Hilb}

\DeclareMathOperator{\Rep}{Rep}

\DeclareMathOperator{\id}{id}

\DeclareMathOperator{\Comm}{Comm}
\DeclareMathOperator{\Hom}{Hom}

\DeclareMathOperator{\rk}{rk}
\DeclareMathOperator{\Stab}{Stab}
\DeclareMathOperator{\supp}{supp}
\DeclareMathOperator{\codim}{codim}
\DeclareMathOperator{\Span}{Span}
\DeclareMathOperator{\PSL}{PSL}

\DeclareMathOperator{\SO}{SO}
\DeclareMathOperator{\Ortho}{O}

\begin{document}

\author{Alexander Thomas}
\title{Generalized Punctual Hilbert Schemes and \texorpdfstring{$\g$}{g}-Complex Structures}
\address{Max-Planck Institute for Mathematics, Vivatsgasse 7, 53111 Bonn, Germany}
\email{athomas@mpim-bonn.mpg.de}

\begin{abstract}
We define and analyze various generalizations of the punctual Hilbert scheme of the plane, associated to complex or real Lie algebras. Out of these, we construct new geometric structures on surfaces whose moduli spaces share multiple properties with Hitchin components, and which are conjecturally homeomorphic to them. For simple complex Lie algebras, this generalizes the higher complex structure from \cite{FockThomas}. For real Lie algebras, this should give an alternative description of the Hitchin-Kostant-Rallis section defined in \cite{HKR-section}.
\end{abstract}

\maketitle

\section*{Introduction}

\subsection*{Motivation.}
The main motivation for this paper is to get a geometric approach to \textit{Hitchin components}. These components were constructed by Nigel Hitchin in his famous paper \cite{Hit.1} using analytic methods (Higgs bundles). Hitchin components are connected components of the character variety $\Hom(\pi_1(\Sigma), G)/G$ where $\Sigma$ is a smooth surface, closed and without boundary and $G$ is a adjoint group of a split real form of a complex simple Lie group (for example $\PSL_n(\R)$).

For the group $G=\PSL_2(\R)$, the Hitchin component is nothing but \textit{Teichm\"uller space}, which has various geometric descriptions, for example as the moduli space of complex structures on $\S$. 

For $G=\PSL_n(\R)$, Vladimir Fock and the author defined in \cite{FockThomas} a new geometric structure, called the \textit{higher complex structure}, whose moduli space shares various properties with Hitchin's component for $\PSL_n(\R)$. 
In \cite{Thomas}, the author proved several steps towards a \textit{canonical diffeomorphism between the moduli space of higher complex structures and the $\PSL_n(\R)$-Hitchin component}, which stays a conjecture despite the progress.
The main ingredient to construct the higher complex structure is the punctual Hilbert scheme of the plane and its zero-fiber.

In this article, we pursue these ideas by defining a $\g$-complex structure, for a complex simple Lie algebra $\g$, using a generalization of the punctual Hilbert scheme. This should give a geometric approach to Hitchin components for any split real group $G$.

Our strategy to define these new objects is twofold: on the one hand we use the various descriptions of the punctual Hilbert scheme, especially the description as variety of commuting matrices, in order to generalize to an arbitrary $\g$. On the other, we got inspiration from Hitchin's original paper \cite{Hit.1} (section 5) where he starts with a principal nilpotent element and deforms it into an element of a principal slice (a generalized companion matrix). Instead of deforming the principal nilpotent element, we add an element which commutes with it. This has the same number of degrees of freedom as the deformation.

One might also ask what happens for real groups which are not split. In the theory of Higgs bundles, there is the notion of a \textit{$G_\R$-Higgs bundle} (or just $G$-Higgs bundle), which was used in \cite{HKR-section} to define a generalization of the Hitchin section, called \textit{Hitchin-Kostant-Rallis section}. A $G_\R$-Higgs bundle canonically gives a connection with monodromy in $G_\R$ (through the non-ablian Hodge correspondence). In the split real case, the Hitchin-Kostant-Rallis section can be identified with Hitchin's component, but in general, it does not give a component.

We give a counterpart of $G_\R$-Higgs bundles in our language of punctual Hilbert schemes using the theory of Kostant-Rallis \cite{KR} and we introduce the notion of a $\g_\R$-complex structure.

\medskip
\subsection*{Results.}
The punctual Hilbert scheme of the plane allows a description in terms of commuting matrices. Imitating this description, we define, for a simple complex Lie algebra $\g$, the \textbf{$\g$-Hilbert scheme} by 
$$\Hilb(\g) = \{(A,B) \in \g^2 \mid [A,B]=0 \text{ + generic condition}\} /G$$
where the generic condition is described below in Definition \ref{maindef}.
Its zero-fiber $\Hilb_0(\g)$ are those pairs of matrices which are nilpotent. The regular part are those pairs, for which at least one element is regular.
The first result of the paper describes the regular part of the zero-fiber (see Corollary \ref{regzerohilb}):
\begin{fakethm}
The regular part of $\Hilb_0(\g)$ is an affine space of dimension $\rk \g$.
\end{fakethm}

We obtain several results in analogy with the classical theory for punctual Hilbert scheme, for example the existence of a Chow map $\Hilb(\g) \rightarrow \h^2/W$ and a description of $\Hilb(\g)$ as a space of ideals. 

Further, the functorial behavior of $\Hilb(\g)$ is analyzed. In particular, we construct an inclusion $\Hilb(\mf{sl}_2) \hookrightarrow \Hilb(\g)$ and a sort of inverse on the level of the regular zero-fiber:
\begin{equation}\label{muintro}
\mu:\Hilb_0^{reg}(\g) \rightarrow \Hilb_0(\mf{sl}_2).
\end{equation}

\medskip
Using the generalized Hilbert scheme, we define a \textbf{$\g$-complex structure} on a surface $\S$ to be a $G$-gauge class of elements of the form $$\Phi_1(z) dz+ \Phi_2(z) d\bar{z} \in \Omega^1(\Sigma, \g) = \Omega^1(\Sigma,\C)\otimes \g$$ such that $$[(\Phi_1(z), \Phi_2(z))] \in \Hilb^{reg}_0(\g)$$ for all $z\in \Sigma$ and some generic constraint (see Definition \ref{def-g-complex-1}).

For $\g=\mf{sl}_2$, we recover the complex structure. For classical $\g$, a $\g$-complex structure is described by \emph{higher Beltrami differentials}.
Using the map $\mu$ from Equation \eqref{muintro} above, we get a first result for $\g$-complex structures (see Proposition \ref{inducedcomplex}):
\begin{fakeprop}
A $\g$-complex structure induces a complex structure on $\S$.
\end{fakeprop}

Using a representation $\g \hookrightarrow \mf{sl}_m$, we define the notion of a \textit{higher diffeomorphism of type $\g$}. We consider $\g$-complex structures modulo these transformations.
For $\g$ of classical type, the local theory is described by Theorem \ref{thm1}:
\begin{fakethm}
For $\g$ of type $A_n$, $B_n$ or $C_n$, any two $\g$-complex structures are locally equivalent under higher diffeomorphism of type $\g$.

For $\g$ of type $D_n$, all $\g$-complex structures with non-vanishing higher Beltrami differential $\sigma_n$ are equivalent under higher diffeomorphisms. However, the zero locus of $\sigma_n$ is an invariant.
\end{fakethm}

The moduli space of $\g$-complex structures, denoted by $\hat{\mc{T}}_\g$ enjoys the following properties (see Theorem \ref{thm2}):
\begin{fakethm}
For $\g$ of classical type, the moduli space $\hat{\mathcal{T}}_{\g}$ is a contractible manifold of complex dimension $(g-1)\dim \g$. Further, there is a copy of Teichm\"uller space inside $\hat{\mathcal{T}}_{\g}$. 
Along this copy of Teichmüller space, the cotangent space at any point $I$ is given by 
$$T^*_{I}\hat{\mathcal{T}}_{\g} = \bigoplus_{m=1}^{r} H^0(K^{m_i+1})$$ where $(m_1,...,m_r)$ are the exponents of $\g$ and $r=\rk \g$ denotes the rank of $\g$.
\end{fakethm}

Note the appearance of the Hitchin base, which serves as parametrization for the Hitchin component. This explains why we think that the moduli space should be canonically homeomorphic to Hitchin's component.

To a point in the cotangent bundle $T^*\hat{\mathcal{T}}_{\g}$, we can associate a \textit{spectral curve}, living in $T^{*\C}\S$. We recover the spectral data of \cite{Hit3} in our setting.

\medskip
For a real Lie algebra $\g_\R$, using the theory of Kostant-Rallis \cite{KR}, we define a \textit{punctual Hilbert scheme associated to $g_\R$}, denoted by $\Hilb(\g_\R)$.
The link to the complex case is given by (see Theorem \ref{split-complex}):
\begin{fakethm}
For any $\g_\R$, there is a map $\Hilb(\g_\R)\rightarrow \Hilb(\g)$. In the case of the split real form, this is an isomorphism on the regular parts.
\end{fakethm}

In the same vein as for a complex Lie algebra, we define a \textit{$\g_\R$-complex structure} and its moduli space. For the split real form, we recover $\hat{\mc{T}}_\g$.

\medskip
\subsection*{Perspectives.}
We wish to give a larger conjectural picture describing the link between Hitchin's component and the moduli space of $\g$-complex structures $\hat{\mathcal{T}}_{\g}$. This motivates the definition of the $\g$-Hilbert scheme. In the case of $\g = \mf{sl}_n$, parts of this large picture are proven in \cite{Thomas}.

Hitchin's original construction in \cite{Hit.1} of components in character varieties uses Higgs bundles and the hyperk\"ahler structure of its moduli space $\mc{M}_H$. In one complex structure, say $I$, $\mc{M}_H$ has the complex structure from Higgs bundles. In all combinations of $J$ and $K$, it is the moduli space of flat $G^{\C}$-connections. The non-abelian Hodge correspondence is equivalent to the twistor description of this hyperk\"ahler manifold. Hitchin constructs a fibration of $\mc{M}_H$ over a space of holomorphic differentials, whose fibers via the non-abelian Hodge correspondence give flat connections with monodromy in the split real group $G$.

There is a similar conjectural picture for $\g$-complex structures: a hyperk\"ahler manifold $\mc{M}$, which in complex structure $I$ is the cotangent space to the moduli space of $\g$-complex structures $T^*\hat{\mc{T}}_{\g}$ and in all combinations of $J$ and $K$ is the moduli space of flat $G^{\C}$-connections. The analogue of Hitchin's fibration is simply the projection $\pi: T^*\hat{\mc{T}}_{\g} \rightarrow \hat{\mc{T}}_{\g}$. One has to prove an analogue of the non-abelian Hodge correspondence, i.e. a deformation of a pair ($\g$-complex structure, set of holomorphic differentials) to flat connections, and that the monodromy of the fibers of the projection $\pi$ lies in the split real group $G$.

The conception behind this analogy is the following: In Hitchin's case, we have a fixed complex structure on $\Sigma$ and a holomorphic Higgs field $\Phi \in H^{(1,0)}(\Sigma,\g)$ which gives a flat connection $\mc{A}(\lambda)=\lambda \Phi +A+ \lambda^{-1}\Phi^*$.
To get the Hitchin section, we choose a principal nilpotent element $f$ in the Lie algebra $\g$ and deform it into an element of the principal slice $f+Z(e)$ where $e$ is the nilpotent element of $\g$ which together with $f$ forms a principal $\mf{sl}_2$-triple and $Z(e)$ denotes its centralizer.

To avoid fixing a complex structure, we start with $\Phi = \Phi_1 dz+\Phi_2 d\bar{z}$. The flatness of $\mc{A}(\lambda)$ gives that $\Phi_1$ and $\Phi_2$ commute. We further impose $\Phi_1$ and $\Phi_2$ to be nilpotent. More specifically, we take $\Phi_1$ to be the principal nilpotent element $f$ and we choose $\Phi_2 \in Z(f)$.
Thus we have the same number of degrees of freedom as in the Higgs bundle setting. A pair of commuting nilpotent matrices of this form is precisely a point in $\Hilb^{reg}_0(\g)$ which we used to construct $\g$-complex structures.

\medskip
\subsection*{Structure.}
The outline of the paper is the following:
\begin{itemize}
	\item[$\bullet$] Section \ref{section1} treats the definition and properties of the $\g$-Hilbert scheme. In particular we describe its regular part, define a Chow morphism, a map to a space of ideals and we invest its topology.
	
	\item[$\bullet$] Section \ref{classical-case} gives explicit descriptions of the $\g$-Hilbert scheme for classical $\g$.
	
	\item[$\bullet$] Section \ref{section2} is devoted to the construction of $\g$-complex structures and basic properties.
	
	\item[$\bullet$] Section \ref{section3} analyzes the moduli space of $\g$-complex structures. In particular, we define a notion of higher diffeomorphisms, study the local theory and give a construction of a spectral curve.
	
	\item[$\bullet$] Section \ref{realcase} generalizes both punctual Hilbert schemes and complex structures to the case of a real Lie algbra $\g_\R$. For the split real form we recover the $\g$-complex structure.
	
	%\item[$\bullet$] Section \ref{section4} gives a larger picture for the conjectural link to Hitchin components.
	
	\item[$\bullet$] Appendix \ref{appendix:A} reviews the main properties of the punctual Hilbert scheme of the plane.
	
	\item[$\bullet$] Appendix \ref{appendix:B} gathers all properties we need in the paper of regular elements in semisimple Lie algebras.
	
	\item[$\bullet$] Appendix \ref{haimancoords} presents Haiman's coordinates on the Hilbert scheme and an apparently new result on its symplectic structure.
	
\end{itemize}

\medskip
\subsection*{Notations.}
Throughout the paper, we denote by $\g$ a complex simple Lie algebra, by $\h$ a Cartan subalgebra, by $W$ its Weyl group and by $G$ its adjoint group (the unique Lie group $G$ with Lie algebra $\g$ with trivial center). For $A\in \g$, we denote by $Z(A)$ its centralizer, i.e. the elements commuting with $A$.
Whenever we speak about real objects (Section \ref{realcase}), we explicitly put an index, for example $\g_\R$ for a real Lie algebra. Similarly, we write $\mf{sl}_2$ for $\mf{sl}_2(\C)$, but for the real Lie algebra, we will always write $\mf{sl}_2(\R)$.

$\Sigma$ denotes a smooth surface, closed, without boundary and orientable. A reference complex coordinate system on $\S$ is denoted by $(z,\bar{z})$, and the induced linear coordinates on $T^{*\C}\S$ are denoted by $(p,\bar{p})$.
The equivalence class of an element $A$  will be written $[A]$.

\medskip
\subsection*{Acknowledgments.}
I wish to express my gratitude towards Vladimir Fock and Oscar Garcìa-Prada for all the fruitful insights.
I also thank Loren Spice and Mykola Matviichuk for helpful comments.
Most of this paper is part of my PhD thesis which I accomplished at the University of Strasbourg. The last section was done during my stay at the Max-Planck Institute for Mathematics in Bonn.

\section{Generalized punctual Hilbert scheme}\label{section1}

In this section, we generalize the punctual Hilbert scheme to a $\g$-Hilbert scheme and explore the properties of the new object. In particular we define a Chow map, and give a description as space of ideals.
The reader not familiar with punctual Hilbert schemes should consult Appendix \ref{appendix:A}.

\subsection{Definitions and first properties}

The punctual Hilbert scheme $\Hilb^n(\C^2)$ has several descriptions:
\begin{itemize}
\item as a space of ideals (the \textit{idealic viewpoint})
\item as a desingularization of the configuration space $\h^2/W$ for $\g=\mf{gl}_n$
\item as a space of commuting matrices (the \textit{matrix viewpoint}). 
\end{itemize}
It is the matrix viewpoint which will be generalized. So let us recall it quickly here:
$$\Hilb^n(\C^2) \cong \{(A,B) \in \mf{gl}_n^2 \mid [A,B]=0, (A,B) \text{ admits a cyclic vector}\} / GL_n.$$

The main difficulty is to find an intrinsic condition which generalizes the existence of a cyclic vector. Here is our proposal:

\begin{definition}\label{maindef}
The generalized punctual Hilbert scheme, or \textbf{$\g$-Hilbert scheme}, denoted by $\Hilb(\g)$, is defined by
$$\Hilb(\g) = \{(A,B) \in \g^2 \mid [A,B]=0, \dim Z(A,B) = \rk \g\} /G$$
where $Z(A,B)$ denotes the common centralizer of $A$ and $B$, i.e. the set of elements $C \in \g$ which commute with $A$ and $B$.
\end{definition}

The condition on the dimension of the common centralizer does not come from nowhere: Proposition \ref{doublecomm} of Appendix \ref{appendix:B} shows that $\rk \g$ is the minimal possible dimension for the centralizer of a commuting pair.
Define the \textit{commuting variety} by $\Comm(\g)=\{(A,B)\in \g^2 \mid [A,B]=0\}$. The $\g$-Hilbert scheme is the set of all regular points of $\Comm(\g)$ modulo $G$.

\begin{Remark}
Ginzburg has defined the notion of a principal nilpotent pair in \cite{Ginzburg}, which is more restrictive than ours. He calls ``nil-pairs'' elements of our $\g$-Hilbert scheme, but he does not investigate them.
\end{Remark}

Let us give two examples of elements in the $\g$-Hilbert scheme:
\begin{example}
Let $A \in \g$ be a regular element. Then by a theorem of Kostant (see \ref{thmKost}), its centralizer $Z(A)$ is abelian. So for any $B\in Z(A)$, we have $Z(A) \subset Z(B)$, thus $Z(A,B) = Z(A) \cap Z(B) = Z(A)$ is of dimension $\rk \g$. Therefore $[(A,B)] \in \Hilb(\g)$.

If $A$ is principal nilpotent, then $B\in Z(A)$ is also nilpotent. So $[(A,B)]\in \Hilb_0(\g)$, the zero-fiber defined below.

If $B=0$ then $[(A,0)]$ is in $\Hilb(\g)$ iff $A$ is regular.
\end{example}

\begin{example}\label{Young}
Let $(A,B)$ be a commuting pair of matrices in $\mf{sl}_n$ admitting a cyclic vector, i.e. an element of the reduced Hilbert scheme. 
One way to get such a pair is the following construction: take a Young diagram (our convention is to put the origin in the upper left corner as for matrices) with $n$ boxes (see Figure \ref{Youngdiag}). Associate to each box a vector of a basis of $\C^n$. Define $A$ to be the matrix which translates to the right, i.e. sends a vector to the vector in the box to the right or to 0 if there is none. Let $B$ be the matrix which translates to the bottom. Then $A$ and $B$ clearly commute and are nilpotent.
In Proposition \ref{cycliccentralizer} below, we show that $Z(A,B)$ is of minimal dimension in that case.

\begin{figure}[H]
\centering
\includegraphics[height=2cm]{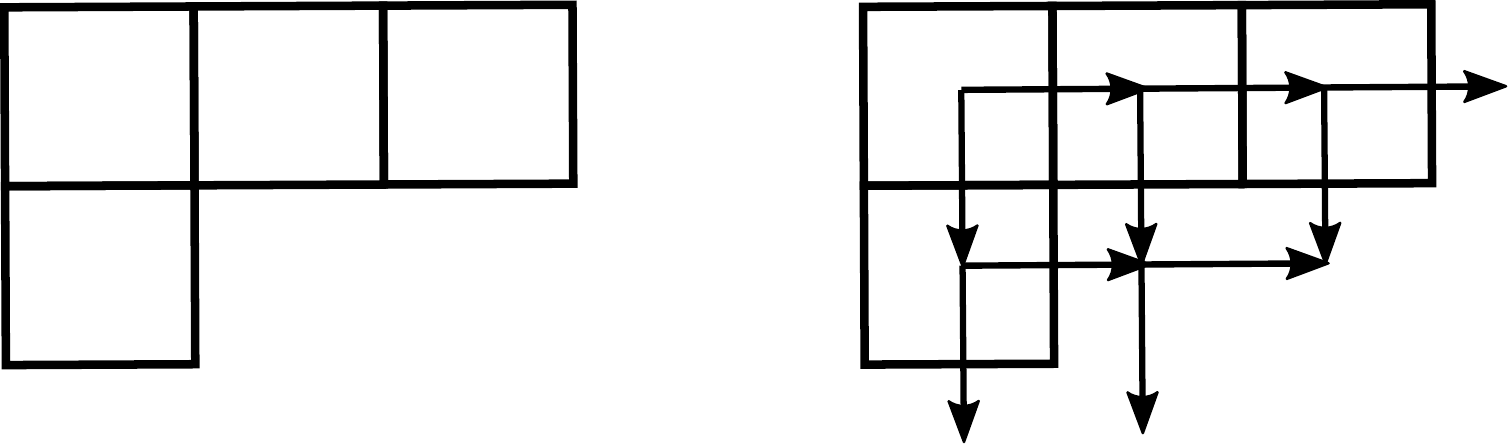}

\caption{Young diagram and commuting nilpotent matrices}
\label{Youngdiag}
\end{figure}
\end{example}

Guided by these examples, we define several subsets of the $\g$-Hilbert scheme and explore their relations.
First, we define the zero-fiber and the regular part which will both play a mayor role in the definition of a $\g$-complex structure. We also define the cyclic part, which is not intrinsically defined since it uses a representation of $\g$. The cyclic part will be used to define a map to a space of ideals, getting a generalization of the original description of the punctual Hilbert scheme.

\begin{definition}\label{partshilb}
The \textbf{zero-fiber} of the $\g$-Hilbert scheme is defined by $$\Hilb_0(\g)=\{[(A,B)] \in \Hilb(\g) \mid A \text{ and } B \text{ nilpotent}\}.$$

We define the \textbf{regular part} of the $\g$-Hilbert scheme, denoted by $\Hilb^{reg}(\g)$, to be those conjugacy classes $[(A,B)]$ in which $A$ or $B$ is a regular element of $\g$.

Finally for classical $\g$, let $\rho$ denote the natural representation of $\g$ (i.e. $\mf{sl}_n \subset \mf{gl}_n, \mf{so}_n \subset \mf{gl}_n$ and $\mf{sp}_{2n} \subset \mf{gl}_{2n}$). Define the \textbf{cyclic part} of the $\g$-Hilbert scheme by 
$$\Hilb^{cycl}(\g) = \{(A,B) \in \g^2 \mid [A,B]=0, (\rho(A),\rho(B)) \text{ admits a cyclic vector}\}/G.$$
\end{definition}

\begin{Remark}
In the definition of the cyclic part, it would be more natural to consider the adjoint representation, but even in the case of $\mf{sl}_2$, this would give a map to a space of ideals, which is not the one of $\Hilb^2_{red}(\C^2)$.

Instead of the standard representation, one could also use a non-trivial representation of minimal dimension, which is defined for all $\g$. For classical $\g$, this is always the standard representation, apart from type $D_3$ and $D_4$. 
\end{Remark}

The first relation between the various Hilbert schemes is the inclusion of the cyclic part in the $\g$-Hilbert scheme, which justifies the name ``cyclic part'':
\begin{prop}\label{cycliccentralizer}
For $\g$ of classical type, we have $\Hilb^{cycl}(\g) \subset \Hilb(\g)$.
\end{prop}
\begin{proof}
Recall $\rho$, the natural representation of $\g$ on $\C^m$. For simplicity, we write $A$ instead of $\rho(A)$ here.

Let $(A,B) \in \g^2$ admitting a cyclic vector $v$. Let $C \in Z(A,B)$. Then $C$ is a polynomial in $A$ and $B$. Indeed, there is $P\in \C[x,y]$ such that $Cv = P(A,B)v$. Since $C$ commutes with $A$ and $B$, we then get for any polynomial $Q$ that $CQ(A,B)v = Q(A,B)Cv = Q(A,B)P(A,B)v= P(A,B)Q(A,B)v$, so $C=P(A,B)$.

Therefore the common centralizer of $(A,B)$ in $\mf{gl}_m$ is $\C[A,B]/I$ where $I=\{P\in \C[x,y] \mid P(A,B) = 0\}$. We know from Appendix \ref{appendix:A} that $I$ is of codimension $m$ since $(A,B)$ admits a cyclic vector. We have $Z(A,B) = Z_{\mf{gl}_m}(A,B) \cap \g$. One can easily check that for $\g$ of type $A_n$, a polynomial $P(A,B)$ is in $\g$ iff its constant term has a specific form, given by the other coefficients (to ensure trace zero). For type $B_n, C_n$ and $D_n$, $P(A,B)$ is in $\g$ iff $P$ is odd. One checks in each case that the dimension of $Z(A,B)$ equals the rank of $\g$.
\end{proof}

In general, the inclusion of the cyclic Hilbert scheme is strict as shows the following example:
\begin{example}
Consider $A=\left(\begin{smallmatrix} 0 & 1 & 0 \\ 0 & 0 & 0 \\ 0 & 0 & 0 \end{smallmatrix} \right)$ and $B=\left(\begin{smallmatrix} 0 & 0 & 1 \\ 0 & 0 & 0 \\ 0 & 0 & 0 \end{smallmatrix}\right)$ in $\mf{sl}_3$. One easily checks that the pair $(A,B)$ does not admit any cyclic vector, but that their common centralizer is of dimension 2. So $[(A,B)] \in \Hilb(\mf{sl}_3)\backslash \Hilb^{cycl}(\mf{sl}_3)$.
\end{example}
This example will be used in Subsection \ref{topology} to show that $\Hilb(\g)$ is not Hausdorff.

In general, there is no link between regular and cyclic part. Example \ref{Young} shows that cyclic elements are not always regular and the following example shows that regular element are not always cyclic:
\begin{example}\label{regnotcyclic}
For $\g$ of type $D_n$, let $f$ be a principal nilpotent element. Then one checks that $[(f,0)] \in \Hilb(\mf{so}_{2n})$ is regular but not cyclic (see also Subsection \ref{Dn}).
\end{example}

Let us turn to the regular part. It turns out that if one fixes a \textit{principal slice} $f+Z(e)$ in $\g$ (see Appendix \ref{appendix:B}), there is a preferred representative for regular classes:
\begin{prop}\label{paramit}
Any class $[(A,B)] \in \Hilb^{reg}(\g)$ where $A$ is regular can uniquely be conjugated to $(A \in f+Z(e), B \in Z(A))$.
\end{prop}
\begin{proof}
By the property of the principal slice, there is a unique conjugate of $A$ which is in the principal slice $f+Z(e)$. Denote still by $A$ and $B$ these conjugates. The only thing to show is that $B$ is unique which is done in the next lemma. 
\end{proof}

\begin{lemma}
If $A\in \g$ is regular, $g\in G$ such that $Ad_g(A)=A$ and $B\in Z(A)$, then $Ad_g(B)=B$.
\end{lemma}
\begin{proof}
By Kostant's theorem \ref{thmKost}, we know that $Z(A)$ is abelian. So the infinitesimal version of the lemma is true. We conclude by the connectedness of the stabilizer of $A$, given by the next lemma.
\end{proof}
\begin{lemma}
For a regular element $A \in \g$, its stabilizer $\Stab(A) = \{g \in G \mid Ad_g(A)=A\}$ in the adjoint group $G$ is connected.
\end{lemma}
\begin{proof}
Decompose $A$ into Jordan form: $A=A_s+A_n$ with $A_s$ semisimple, $A_n$ nilpotent and $[A_s,A_n]=0$. So $A_n\in Z(A_s)$. The structure of the centralizer $Z(A_s)$ is well-known: it is a direct sum of a Cartan $\h$ containing $A_s$ with all root spaces $\g_{\alpha}$ where $\alpha$ is a root such that $\alpha(A_s)=0$. It is also known that $Z(A_s)$ is reductive, so a direct sum $Z(A_s)=\mf{c}\oplus \g_s$ where $\mf{c}$ is the center and $\g_s$ is the semisimple part of $Z(A_s)$. In particular the center $\mf{c}$ is included in $\h$. So $A_n \in \g_s$ since $A_n$ is nilpotent. Denote by $G_s$ the Lie group with trivial center with Lie algebra $\g_s$.

We know that $A$ is regular is equivalent to $A_n$ being regular nilpotent in $\g_s$ (see \cite{Kost2}, proposition 0.4). 
We also know that the $G$-equivariant fundamental group of the orbit of $A$ (which is the space of connected components of $\Stab(A)$) is the same as the $\Stab(A_s)$-equivariant fundamental group of the $\Stab(A_s)$-orbit of $A_n$ (see Proposition 6.1.8. of \cite{Coll} adapted to the adjoint group). In other words, the connected components of $\Stab_G(A)$ are the same as the connected components of $\Stab_{G_s}(A_n)$ since the $\Stab(A_s)$-orbit of $A_n$ is equal to the $G_s$-orbit of $A_n$.

So we are reduced to the principal nilpotent case. Using the classification of simple Lie algebras, one can check explicitly in Collingwood-McGovern's book \cite{Coll} the tables 6.1.6. for classical $\g$ and the tables at the end of chapter 8 for exceptional $\g$ that the stabilizer of a principal nilpotent element is always connected.
\end{proof}
\begin{Remark}
It is surprising that the last lemma has never been stated (at least not to our knowledge). It would be interesting to find a direct argument, without using the classification of simple Lie algebras.
\end{Remark}

\begin{coro}\label{regzerohilb}
The regular zero-fiber $\Hilb^{reg}_0(\g) = \Hilb^{reg}(\g)\cap \Hilb_0(\g)$ is an affine space of dimension $\rk \g$.
\end{coro}
\begin{proof}
This follows directly from the previous proposition using the fact that $A \in f+Z(e)$ is nilpotent iff $A=f$. So $\Hilb^{reg}_0(\g)$ is described by $Z(f)$ which is a vector space of dimension $\rk \g$.
\end{proof}

We know that both the regular and the cyclic part are in general strictly included in the $\g$-Hilbert scheme. But they are dense subspaces:
\begin{prop}\label{density}
The regular part $\Hilb^{reg}(\g)$ is dense in $\Hilb(\g)$. For classical $\g$, the cyclic part is also dense in $\Hilb(\g)$.
\end{prop}
\begin{proof}
By a theorem of Richardson (see \ref{Richardson}), the set of semisimple commuting pairs is dense in the commuting variety $\Comm(\g)$. So the set of semisimple regular elements is also dense in $\Comm(\g)$. Passing to the quotient by $G$, we get that the classes of semisimple regular pairs are dense in $\Hilb(\g)$ since $\Hilb(\g)\subset \Comm(\g)/G$ and all semisimple regular pairs are in $\Hilb(\g)$. Since the semisimple regular pairs are in the regular part, we get the density of $\Hilb^{reg}(\g)$ in $\Hilb(\g)$.

For classical $\g$, we have the same argument for the cyclic part since semisimple regular pairs are cyclic.
\end{proof}

To end the section, we state an analogue of Kostant's theorem about abelian subalgebras of centralizers:
\begin{prop}
For any commuting pair $(A,B) \in \Comm(\g)$, there is an abelian subspace of dimension $\rk \g$ in the common centralizer $Z(A,B)$.
\end{prop}
\begin{proof}
The proof is completely analogous to Kostant's proof for Theorem \ref{thmKost}: we use a limit argument. Let $(A_n,B_n)$ be a sequence of regular semisimple pairs converging to $(A,B)$ (exists since regular semisimple pairs are dense). We know that $Z(A_n,B_n)$ is a $\rk \g$-dimensional abelian subspace of $\g$. Since the Grassmannian $Gr(\rk \g, \dim \g)$ is compact, there is a subsequence of $Z(A_n,B_n)$ which converges. It is easy to prove that the limit is included in $Z(A,B)$ and is commutative.
\end{proof}
\begin{coro}
For $[(A,B)] \in \Hilb(\g)$, the common centralizer $Z(A,B)$ is abelian.
\end{coro}
\begin{Remark}
For classical $\g$, the corollary is easy for the cyclic part since $Z(A,B) = \C[x,y]/I \cap \g$ which is abelian since $\C[x,y]$ is.
\end{Remark}

In the following sections, we generalize as far as possible the other viewpoints of the usual Hilbert scheme (resolution of configuration space and idealic viewpoint) to our setting.

\subsection{Chow map}

We want to generalize the Chow map, which goes from $\Hilb^n(\C^2)$ to the configuration space (see \ref{resofsing}).

Fix a Cartan subalgebra $\h$ in $\g$. Recall the Jordan decomposition in a semisimple Lie algebra: for $x \in \g$, there is a unique pair $(x_s, x_n)$ with $x=x_s+x_n$, $x_s$ semisimple, $x_n$ nilpotent and $[x_s,x_n]=0$. For a semisimple element $x$, denote by $x^*$ a conjugate in the Cartan $\h$ (unique up to $W$-action).

The \textbf{Chow map}  $ch: \Hilb(\g) \rightarrow \h^2/W$ is defined by $$ch([(A,B)]) = [(A_s^*,B_s^*)]$$ where the brackets $[.]$ denotes the equivalence class. For semisimple regular pairs, this map corresponds to a simultaneous diagonalization.

\begin{prop}
The Chow map $ch$ is well-defined and continuous.
\end{prop}
\begin{proof}
Since $[A,B]=0$, we also have $[A_s,B_s]=0$ by a simultaneous Jordan decomposition in a faithful representation. Hence there is a conjugate of the pair $(A_s,B_s)$ which lies in $\h^2$. Since the adjoint action of $G$ on $\g$ restricts to the $W$-action on $\h$, the map $ch$ is well-defined.

The map $x\mapsto x_s^*$ is continuous which simply follows from the continuity of eigenvalues. Hence the Chow map is continuous as well.
\end{proof}
\begin{Remark}
The Jordan decomposition $x\mapsto (x_s,x_n)$ is not continuous at all, since semisimple elements are dense in $\g$ for which we have $x_n=0$ and for all non-semisimple elements we have $x_n\neq 0$. But the map $x\mapsto x_s$ is continuous.
\end{Remark}

This map permits to think of a generic element of $\Hilb(\g)$ as a point in $\h^2/W$, or via a representation of $\g$ on $\C^m$, as a set of $m$ points in $\C^2$ with a certain symmetry. For $\g=\mf{sl}_n$ for example, these are $n$ points with barycenter 0.

Since $\Hilb(\g)$ is even not Hausdorff (see Subsection \ref{topology}), it cannot be a non-singular variety. Nevertheless we conjecture the following:

\begin{conj}\label{conj1}
There is a modified version of $\Hilb(\g)$, identifying some points, which is a smooth projective variety such that the Chow morphism is a resolution of singularities.
\end{conj}

\subsection{Idealic map}\label{idealic}

In this subsection, $\g$ is a classical Lie algebra. In that case we can associate to any regular element of the $\g$-Hilbert scheme an ideal, which we call \textit{idealic map}. Recall the standard representation $\rho$ of $\g$ on $\C^m$ (see Definition \ref{partshilb}). We will write $A$ instead of $\rho(A)$.

We wish to define a map like in \ref{bijhilbert}: \begin{equation}\label{ideal}[(A,B)] \mapsto I(A,B)=\{P\in \C[x,y] \mid P(A,B) = 0\}.\end{equation} 
If $[(A,B)] \in \Hilb^{cycl}(\g)$ is cyclic, this ideal is of codimension $m$. But if the pair is not cyclic, there is no reason why the codimension should be $m$. In fact, there are examples for $\g$ of type $D_n$ where the codimension is smaller. 

We wish the idealic map to be continuous, so $I$ has to be of constant codimension. A strategy would be to define the idealic map $I$ on the cyclic part $\Hilb^{cycl}(\g)$ (which is dense by Proposition \ref{density}) and to extend it by continuity. Unfortunately, the map can not be extended in a continuous way as shown in the following example:
\begin{example}\label{idealnotcont}
Take $\g$ of type $D_n$. Denote by $f$ a principal nilpotent element. The pair $[(f,0)] \in \Hilb(\mf{so}_{2n})$ is not cyclic (seen in Example \ref{regnotcyclic}). Using the matrix $S$ defined in Equation \eqref{matrixS}, we can approach $(f,0)$ by $(f,tS)$ or by $(f+tS^\top,0)$ for $t\in\C^{\times}$ going to 0. These pairs are all cyclic. In the first case, the ideal is $I=\langle x^{2n-1}, xy, y^2=t^2x^{2n-2}\rangle$ which converges as $t$ goes to 0 to $\langle x^{2n-1}, xy, y^2\rangle$. In the second case, the ideal is $I=\langle x^{2n}+t^2, y \rangle$ converging to $\langle x^{2n},y \rangle$.
\end{example}

Because of this difficulty, our strategy is to define a space of ideals $I_{\g}(\C^2)$, then a map $\Hilb^{cycl}(\g)\rightarrow I_{\g}(\C^2)$ and to extent it over the regular part $\Hilb^{reg}(\g)$ (in a non-continuous way). The last step is only necessary for $\g$ of type $D_n$ since for the other classical types the regular part is included in the cyclic part as we will see in the sequel. The extension for $D_n$ will be defined \textit{ad hoc} in Subsection \ref{Dn}.

The previous section taught us to think of a generic element of $\Hilb(\g)$ as a $m$-tuple of points in $\C^2$ invariant under the Weyl group $W$. For type $A_n$ this means that the barycenter of the points is the origin. For the other classical types, this means that the set of points is symmetric with respect to the origin. Thus the defining ideal of these points is also invariant under the action of $W$. Hence the following definition.

\begin{definition}
We define the \textbf{space of ideals} of type $\g$, denoted by $I_{\g}(\C^2)$, to be the set of ideals in $\C[x,y]$ which are of codimension $m$ and $W$-invariant. For type $B_n, C_n$ and $D_n$ this means that $I$ is invariant under $(x,y)\mapsto (-x,-y)$.
\end{definition}

The map $I: \Hilb^{cycl}(\g)\rightarrow I_{\g}(\C^2)$ given by Equation \eqref{ideal} above is well-defined. Indeed, the codimension is $m$ by cyclicity and the ideal is $W$-invariant since this is a closed condition and it is true on the dense subset of regular semisimple pairs.

Notice that $I_{\g}(\C^2)$ is the same for $\g$ of type $C_n$ or $D_n$. But we will see that the idealic map $I$ has not the same image in the two cases.
We will also see that for $\g$ of type $A_n, B_n$ or $C_n$ the idealic map is injective. But for type $D_n$ it is not (it is generically 2 to 1). This comes from the fact that the Weyl group acting on the generic $2n$ points, coming in $n$ pairs $(P_i, P_{i+1}=-P_i)$, cannot exchange $P_1$ and $P_2$ while leaving all other points fixed.

As for the usual Hilbert scheme, there is a direct link between the idealic map and the Chow morphism:
\begin{prop}
The Chow map $ch$ is the composition of the idealic map with the map which associates to an ideal its support, seen as an element of $\h^2/W$: 
$$ch([(A,B)]) = \supp I(A,B).$$
\end{prop}
\begin{proof}
The statement is true on regular semisimple pairs which is a dense subset. For $\g$ of type $A_n$, $B_n$ and $C_n$, it follows by continuity of both the Chow map and the idealic map. For $D_n$, our definition of the idealic map is to pick one of the various possible limits. In particular, the support of the ideal is still given by the Chow map.
\end{proof}

\subsection{Morphisms}\label{mu2}
In this subsection, we analyze the functorial behavior of the $\g$-Hilbert scheme. In particular we construct two maps linked to the zero-fiber of the Hilbert scheme of $\mf{sl}_2$ which will lead in the construction of the moduli space $\hat{\mathcal{T}}_{\g}\Sigma$ of $\g$-complex structures to maps from and to Teichm\"uller space.

Let $\psi: \g_1 \rightarrow \g_2$ be a morphism of Lie algebras. For $[(A,B)] \in \Hilb(\g_1)$, we can associate $[(\psi(A), \psi(B))]$ which is a well-defined map to $\Comm(\g_2)/G_2$. But there is no reason why $\dim Z(\psi(A), \psi(B))$ should be minimal.

If we accept Conjecture \ref{conj1}, that there is a modified version of the $\g$-Hilbert scheme which is a resolution of $\h^2/W$, we have a functorial behavior:
\begin{prop}
Assuming Conjecture \ref{conj1}, there is an induced map $\Hilb(\g_1) \rightarrow \Hilb(\g_2)$.
\end{prop}
\begin{proof}
Choose Cartan subalgebras $\h_1$ and $\h_2$ such that $\psi(\h_1)=\h_2$. Consider the composition $\h_1^2 \rightarrow \h_2^2 \rightarrow \h_2^2/W_2$ using $\psi$ for the first arrow. Since $\psi$ induces a homomorphism between the Weyl groups, we can factor the composition to get a map $\h_1^2/W_1 \rightarrow \h_2^2/W_2$. Finally, consider the composition $\Hilb(g_1) \rightarrow \h_1^2/W_1 \rightarrow \h_2^2/W_2$ where the first arrow comes from the minimal resolution. This is a continuous map and by the universal property of a minimal resolution, the map lifts to $\Hilb(\g_1) \rightarrow \Hilb(\g_2)$.
\end{proof}

Let us study this induced map in the case of the reduced Hilbert scheme, which is a minimal resolution (see Appendix \ref{appendix:A}). Take $\psi:\mf{sl}_m\rightarrow \mf{sl}_n$  inducing a map $\Hilb^m_{red}(\C^2)\rightarrow \Hilb^n_{red}(\C^2)$. In the matrix viewpoint, this map is not given by $[(\psi(A),\psi(B))]$. Consider for example the map $\psi: \mf{sl}_2\rightarrow \mf{sl}_4$ given on the standard generators $(e,f,h)$ of $\mf{sl}_2$ by 
$$\psi(e)=\left(\begin{smallmatrix} 0&&&1 \\&0&&\\&&0&\\&&&0\end{smallmatrix}\right), \psi(f)=\left(\begin{smallmatrix} 0&&& \\&0&&\\&&0&\\1&&&0\end{smallmatrix}\right) \text{ and } \psi(h)=\left(\begin{smallmatrix} 1&&& \\&0&&\\&&0&\\&&&-1\end{smallmatrix}\right).$$
 The element $[(h,0)] \in \Hilb^2_{red}(\C^2)$ corresponds to the ideal $I=\langle x^2-1,y\rangle$ which through $\psi$ goes to $\langle x^4-x^2,y\rangle$ which in turn gives the matrices $[(M,0)]$ where $M=\left(\begin{smallmatrix} 1&&& \\&0&1&\\&0&0&\\&&&-1\end{smallmatrix}\right)$.
This is not $[(\psi(h), \psi(0))]$.
It would be interesting to describe the induced map in the matrix viewpoint.

\medskip
Despite this complication, there are two cases where a map between $\g$-Hilbert schemes naturally exists.

The first one is linked to the principal map $\psi: \mf{sl}_2 \rightarrow \g$ which induces a map 
\begin{equation}\label{teichcopy} \Hilb(\mf{sl}_2) \rightarrow \Hilb^{reg}(\g). \end{equation} 
Indeed, any non-zero element of $\mf{sl_2}$ is regular and cyclic. So if $[(A,B)] \in \Hilb(\mf{sl}_2)$ such that $A$ is non-zero, there is by Proposition \ref{paramit} a unique representative $(f+te,B\in Z(e+tf))$ where $(e,f,h)$ denotes the standard generators of $\mf{sl}_2$ and $t\in \C$. So the image is $[(\psi(f)+t\psi(e),\psi(B))]$. Since $(\psi(e), \psi(f), \psi(h))$ is a principal $\mf{sl}_2$-triple (property of the principal map), we know that $\psi(f)+t\psi(e)$ is in the principal slice, thus it is regular, so we land in $\Hilb^{reg}(\g)$.

The second one is a sort of inverse map to the first one, but only on the level of the zero-fiber. Given $[(A,B)]\in \Hilb^{reg}_0(\g)$ where $A$ is regular, there is a principal $\mf{sl}_2$-subalgebra $\mc{S}$ with $A$ as nilpotent element. There is no reason why $B$ should be in $\mc{S}$ but there is a ``best approximation'' in the following sens:

\begin{prop}\label{mu2prop}
Let $A$ be a principal nilpotent element and $B \in Z(A)$. Then there is a unique $\mu_2 \in \C$ such that $B-\mu_2 A$ is not regular.
\end{prop}
\begin{proof}
The strategy of the proof is to use Proposition \ref{prinnilp} of the appendix which characterizes principal nilpotent elements $x$ as those nilpotent elements whose coefficients $x_\alpha$ in the root vector basis $(e_\alpha)$ are non-zero for all simple roots $\alpha$. 

Let $R$ be a root system in $\h^*$ and denote by $R_+$ and $R_s$ the positive and respectively the simple roots.
We can conjugate $A$ to the element given by $\sum_{\alpha \in R_s}e_\alpha$.
The proposition is then equivalent to the statement that $B_{\alpha} = B_{\alpha'}$ for all simple roots $\alpha$ and $\alpha'$.

For two simple roots $\alpha$ and $\alpha'$ such that $\alpha+\alpha' \in R$, using $[A,B]=0$ and $A=\sum_{\alpha \in R_s}e_\alpha$ we get:
$$0=[A,B]_{\alpha+\alpha'}= A_{\alpha}B_{\alpha'}-A_{\alpha'}B_{\alpha} = B_{\alpha'}-B_{\alpha}.$$

Since $\g$ is simple, its Dynkin diagram is connected, so $B_\alpha = B_{\alpha'}$ for all simple roots. The common value $\mu_2$ gives the unique complex number such that $B-\mu_2 A$ is not regular.
\end{proof}

With this proposition, we can now define a map \begin{equation}\label{mu} \mu: \Hilb^{reg}_0(\g) \rightarrow \Hilb_0(\mf{sl}_2)\end{equation} given by $\mu([(A,B)])=[(e,\mu_2 e)]$ or $[(\mu_2 e, e)]$ depending whether $A$ or $B$ is regular. 

An equivalent way to define the map $\mu$ is the following: we can use the previous proposition \ref{mu2prop} to show that the centralizer $Z(A)$ of a principal nilpotent element is a direct product $$Z(A) = \Span(A) \times Z(A)^{irreg}$$ where $Z(A)^{irreg}$ denotes the irregular elements of $Z(A)$. The map $\mu$ is nothing but the projection to the first factor.

\begin{Remark}
We can describe the regular part of the $\g$-Hilbert scheme $\Hilb^{reg}(\g)$ as those classes $[(A,B)]$ such that $\Span(A,B)$ intersects the regular part $\g^{reg}$ non-trivially. This description is more symmetric since it does not prefer $A$ or $B$. From Proposition \ref{mu2prop} we see that the intersection of $\Span(A,B)$ with $\g^{reg}$ is the whole two-dimensional $\Span(A,B)$ from which we have to take out a line. Hence, the intersection has two components. 
\end{Remark}

\subsection{Topology of \texorpdfstring{$\g$}{g}-Hilbert schemes}\label{topology}

It is clear that $\Hilb(\g)$ is a topological space, as a quotient of a subset of $\g^2$.
In this section, we explore this topology of $\Hilb(\g)$, especially for $\g=\mf{sl}_n$. We then formulate some conjectures on its general structure.

For $\g=\mf{sl}_2$, every non-zero element $A\in \g$ is regular and cyclic. Since the centralizer of the pair $(0,0)$ is all of $\mf{sl}_2$, this pair is not in $\Hilb(\mf{sl}_2)$. Thus we have $\Hilb(\mf{sl}_2) = \Hilb^{cycl}(\mf{sl}_2) = \Hilb^2_{red}(\C^2)$ which is a smooth projective variety.

For $\g=\mf{sl}_3$, a detailed analysis, putting $A$ into Jordan normal form, shows that $(A,B)$ has minimal centralizer and is not cyclic iff it is conjugated to a pair $P_1(b):=\left(\left(\begin{smallmatrix} 0&1&0 \\ 0&0&0 \\ 0&0&0\end{smallmatrix}\right), \left(\begin{smallmatrix} 0&b&1 \\ 0&0&0 \\ 0&0&0\end{smallmatrix}\right)\right)$. So 
$$\Hilb(\mf{sl}_3) = \Hilb^3_{red}(\C^2) \cup \{P_1(b) \mid b\in\C\}.$$
At first sight, the topology seems to be a smooth variety (the reduced Hilbert scheme) and a complex line. But a closer look shows that each point of the extra line is infinitesimally close to a point in the variety, meaning that these two points cannot be separated by open sets, infringing the Hausdorff property. The pair $P_1(b)$ is infinitesimally close to $P_2(b):=\left(\left(\begin{smallmatrix} 0&1&0 \\ 0&0&0 \\ 0&0&0\end{smallmatrix}\right),\left( \begin{smallmatrix} 0&b&0 \\ 0&0&0 \\ 0&1&0\end{smallmatrix}\right)\right)$. Indeed any neighborhood of the first pair $P_1(b)$ contains $\left(\left(\begin{smallmatrix} 0&1&0 \\ 0&0&0 \\ 0&0&0\end{smallmatrix}\right), \left(\begin{smallmatrix} 0&b&1 \\ 0&0&0 \\ 0&s&0\end{smallmatrix}\right)\right)$ for some small $s\in \C$ which is conjugated to $\left(\left(\begin{smallmatrix} 0&1&0 \\ 0&0&0 \\ 0&0&0\end{smallmatrix}\right),\left( \begin{smallmatrix} 0&b&s \\ 0&0&0 \\ 0&1&0\end{smallmatrix}\right)\right)$ which lies in a neighborhood of the second pair$P_2(b)$.
Since the idealic map is continuous and for $\mf{sl}_n$ injective on the cyclic part, there cannot be another point of the cyclic part which is infinitesimally close to the first pair $P_1(b)$. Finally, two elements of the extra line can be separated by open sets.
Hence, the space $\Hilb(\mf{sl}_3)$ is obtained from a smooth variety by adding ``double points'' (here in the sens of infinitesimally close points) along a complex line.

Since the idealic map is injective on the cyclic part $\Hilb^{cycl}(\mf{sl}_n)$, the same analysis holds for $\mf{sl}_n$, i.e. $\Hilb(\mf{sl}_n)$ is obtained from a smooth variety (the reduced Hilbert scheme) by adding double points.

There should exist a procedure, like a GIT quotient, giving a modified $\g$-Hilbert scheme which is a Hausdorff space. The GIT quotient does not apply here since $\{(A,B)\in \g^2 \mid [A,B]=0, \dim Z(A,B)=\rk \g\}$ is not a closed variety. In the language of GIT quotients, the pairs $P_1$ and $P_2$ above are both semistable, but there is no polystable element in their closure. 

To give a feeling on what happens, consider the action of $\R_{>0}$ on $\R^2 \backslash \{(0,0)\}$ given by $\lambda.(x_1, x_2)=(\lambda x_1, \lambda^{-1} x_2)$. The orbits are drawn in Figure \ref{nonhaus}.
The quotient space is a set of two lines $L_1$ and $L_2$ with origins $O_1$ and $O_2$ together with two extra points $O_3$ and $O_4$ such that the pairs $(O_1, O_3), (O_1, O_4), (O_2, O_3)$ and $(O_2, O_4)$ are infinitesimally close points (the four points $O_i$ correspond to the four half-axis). In the figure, the dashed lines indicate infinitesimally close points.
From the GIT perspective, all points are semistable (take the constant function 1), the four half-axis are semistable and all other orbits are stable. The orbits of the half-axis are closed in $\R^2 \backslash \{(0,0)\}$ so they should be polystable, but in the quotient the points are still infinitesimally close.
\begin{figure}[h]
\centering
\includegraphics[height=4cm]{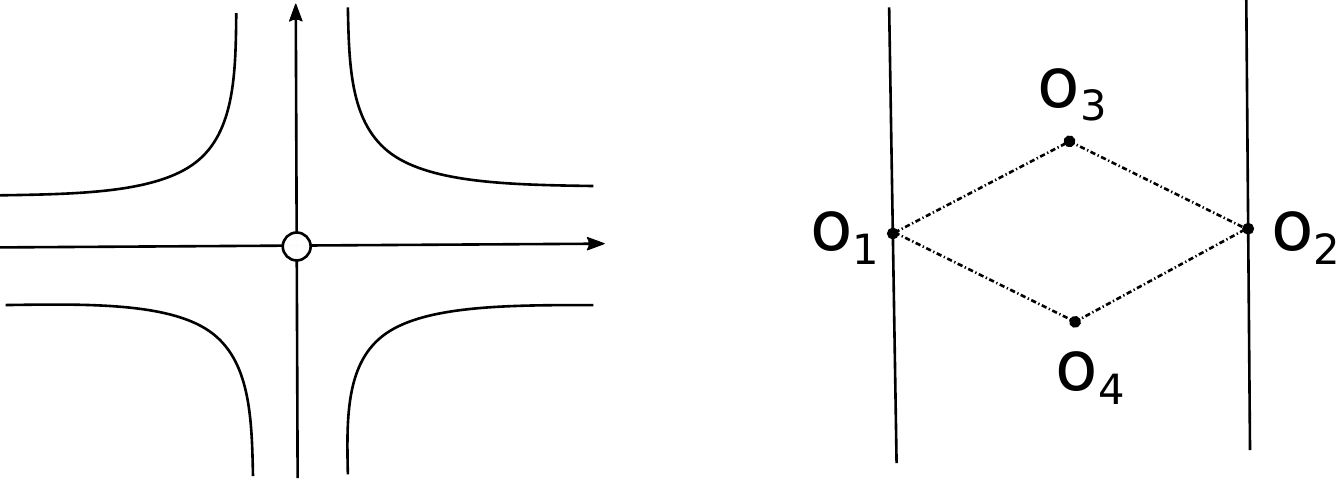}

\caption{Non-Hausdorff quotient}\label{nonhaus}
\end{figure}

We conjecture the following:
\begin{conj}\label{git-conj}
There is a generalized GIT quotient procedure identifying infinitesimally close points in $\Hilb(\g)$, giving a modified $\g$-Hilbert scheme which is Hausdorff, and even smooth. 
\end{conj}
In particular one should find the reduced Hilbert scheme for $\g=\mf{sl}_n$. See also Conjecture \ref{conj1} for a modified $\g$-Hilbert scheme as a resolution of $\h^2/W$.

Assume a smooth version of the $\g$-Hilbert scheme exists. In the $\mf{sl}_n$-case the reduced Hilbert scheme is covered by charts parametrized by partitions of $n$, which also parametrizes nilpotent orbits of $\mf{sl}_n$. For $\g$ of classical type, the nilpotent orbits are parametrized by special partitions (see \cite{Coll}, chapter 5). In general, we conjecture the following for the $\g$-Hilbert scheme:

\begin{conj}\label{charts-by-nilpotent}
The smooth version of $\Hilb(\g)$ is covered by charts parametrized by nilpotent orbits and all these charts are necessary.
 
In particular for classical $\g$, we conjecture that the modified version of $\Hilb_0(\g)$ is isomorphic to the space of ideals of $\C[x,y]$ which are of codimension $m$, $W$-invariant, supported at 0 and which lie in a chart associated to a partition of type $\g$.
\end{conj}

In particular, for every nilpotent $A\in \g$, there has to be an element in $\Hilb(\g)$ containing the conjugacy class of $A$. More precisely, we conjecture:

\begin{conj}
Let $\g$ be of rank at least 3. For a nilpotent element $A\in \g$, there is $B\in Z(A)$ nilpotent such that $\dim Z(A,B) = \rk \g$, i.e. $[(A,B)] \in \Hilb_0(\g)$. This should be true for a generic element $B \in Z(A)$.
\end{conj}

For $\mf{sl}_n$ the conjecture is true: we can associate to a nilpotent element $A$ a partition $\nu$. To the transpose partition $\nu^\top$ (using the transpose of the Young diagram) correspond a nilpotent element $B$ which satisfies the requirements since $(A,B)$ is cyclic. An equivalent way is to use Example \ref{Young} to produce $B$.
%Another way to see it: let $D$ be a Young diagram associated to the partition $\nu$. One can choose a basis and a filling of $D$ by the basis such that the matrix $A$ is the right translation, sending a vector in a box to the vector in the box to the right or to 0 if there is none. Then the matrix $B$ which corresponds to translation to the bottom is nilpotent, commutes with $A$ and the pair $(A,B)$ is cyclic, so $\dim Z(A,B)$ is minimal.

For $\g=\mf{sp}_{4}$ of type $C_2$, there is the following counterexample. That is why we formulate the conjecture only for Lie algebras of rank at least 3.
Take the nilpotent element 
$$A = \left(\begin{array}{c|c}
  0  & \id \\
\hline
  0 & 0
\end{array}\right).$$
Its centralizer is given by
$$Z(A) = \left(\begin{array}{cc|cc}
  0&b  & x&y \\
	-b&0 &y&z \\\hline
  0 & 0 &0 & b\\
	0&0&-b&0
\end{array}\right).$$
An element $B$ of the centralizer is nilpotent iff $b= 0$. In that case the common centralizer $Z(A,B)$ is at least of dimension 3, so $[(A,B)]$ is not in $\Hilb_0(\mf{sp}_4)$.

In general, we cannot hope to find $B\in Z(A)$ such that $(A,B)$ is cyclic. For example take $\g=\mf{sp}_{16}$ and $A$ a nilpotent element corresponding to the partition $[7,5,3,1]$ of 16. If there is $B\in Z(A)$ nilpotent and such that $(A,B)$ is cyclic, there would be an ideal $I$ of codimension 16 whose associated matrices are $A$ and $B$ (see Example \ref{Young}). Using $A$, we see that $I$ has to be of the form 
$$I=\langle x^4, x^3y, x^2y^3, xy^5, y^7=Q(x,y) \rangle$$ where $Q$ is a polynomial with monomial terms in the Young diagram $D$. A partition of type $C_n$ has all odd parts with even multiplicity and one can check that for all choices of the polynomial $Q$, the ideal $I$ is never in a chart with all odd parts with even multiplicity.

\section{Hilbert schemes for classical Lie algebras}\label{classical-case}

In this section, we study the regular part $\Hilb^{reg}(\g)$ and its zero-fiber case by case for classical $\g$.

\subsection{Case \texorpdfstring{$A_n$}{An}}
Consider $\g=\mf{sl}_n$ (of type $A_{n-1}$). We describe first $\Hilb^{reg}_0(\mf{sl}_n)$, its idealic map and then $\Hilb^{reg}(\mf{sl}_n)$ using Proposition \ref{paramit}.

Fix the following principal nilpotent element (with 1 on the diagonal line just under the main diagonal):
$$f=\begin{pmatrix} & & & \\ 1& & & \\ &\ddots & & \\ & &1 & \end{pmatrix}.$$

This element $f$ is cyclic, so we know from \ref{cycliccentralizer} that the centralizer is given by polynomials: 
$Z(f)=\{\mu_2 f+\mu_3 f^2+...+\mu_n f^{n-1}\}$.
So an element of $\Hilb^{reg}_0(\mf{sl}_n)$ can be represented by $(f,Q(f))$ where $Q$ is a polynomial without constant term of degree at most $n-1$. The coefficients $\mu_i$ are called \textit{higher Beltrami coefficients}.

Since here we have $\Hilb^{reg}_0(\mf{sl}_n) \subset \Hilb^{cycl}(\mf{sl}_n)$ (already $f$ is cyclic), the idealic map is given by $$I(f,Q(f))=\{P\in \C[x,y] \mid P(f,Q(f))=0\} = \langle x^n, -y+Q(x) \rangle.$$ We recognize the big cell of the zero-fiber of the punctual Hilbert scheme.

To describe the whole regular part $\Hilb^{reg}(\mf{sl}_n)$, we take the following principal slice given by companion matrices:
$$\begin{pmatrix} & & & t_n\\ 1& & & \vdots\\ &\ddots & &t_2\\ & &1 & \end{pmatrix}.$$

Let $A$ be a matrix of companion type. Notice that the characteristic polynomial of a companion matrix is given by $x^n+t_2x^{n-2}+...+t_n$. 
Since $A$ is still cyclic, its centralizer consists of polynomials in $A$ with constant term determined by the other coefficients (in order to ensure trace zero). Thus, a representative of $\Hilb^{reg}(\mf{sl}_n)$ is given by $(A,B=Q(A))$.

The idealic map is thus given by 
$$I(A,B)=\langle x^n+t_2x^{n-2}+...+t_n, -y+\mu_1+\mu_2x+...+\mu_nx^{n-1} \rangle$$ where $\mu_1$ is given by $\mu_1=\sum_{k=2}^{n-1}\frac{k}{n}t_k\mu_{k+1} \mod t^2$.
One recognizes the big cell of the reduced punctual Hilbert scheme. Notice that the idealic map is injective here.

\subsection{Case \texorpdfstring{$B_n$}{Bn}}\label{Bn}

Consider $\g=\mf{so}_{2n+1}$. Represent $\g$ on $\C^{2n+1}$ using the metric given by $g(e_i,e_j)=\delta_{i,n-j}$ (where $e_i$ are standard vectors), i.e. $g=\left(\begin{smallmatrix} & & 1\\ & \udots & \\ 1& & \end{smallmatrix}\right)$.
A matrix $A$ is in $\g$ iff $\sigma(A)=-A$ where $\sigma$ is the involution consisting in a reflection along the anti-diagonal. In other words $A \in \g$ iff $A_{i,j}=-A_{n+1-j,n+1-i}$ for all $i,j$.

We fix the following principal nilpotent element:
$$f=\begin{pmatrix} &&&&&& \\ 1&&&&&& \\ &\ddots &&&&&\\ &&1&&&&\\ &&&-1&&& \\ &&&& \ddots &&\\ &&&&&-1&\end{pmatrix}.$$

This element is cyclic, so its centralizer by \ref{cycliccentralizer} consists of all odd polynomials: $Z(f)=\{\mu_2f+\mu_4f^3+...+\mu_{2n}f^{2n-1}\}$. A representative of $\Hilb^{reg}_0(\g)$ is thus given by $(f,Q(f))$ where $Q$ is an odd polynomial of degree at most $2n-1$. The coefficients $\mu_{2i}$ are called the higher Beltrami coefficients for $B_n$.

A principal slice is given by $$\begin{pmatrix} &&&&&t_{2n}& \\ 1&&&&\udots && -t_{2n}\\ &\ddots &&t_2&&\udots& \\  &&1&&-t_2&&\\ &&&-1&&& \\ &&&& \ddots &&\\ &&&&&-1&\end{pmatrix}.$$
Let $A$ be a matrix of this type. Its characteristic polynomial is given by $x^{2n+1}-2t_2x^{2n-1}+2t_4x^{2n-3}\pm ... +(-1)^n\times 2t_{2n}x$. So we can really think of the principal slice as a generalized companion matrix. Changing slightly $t_{2i}$ we can get rid of signs and the factor 2 in the characteristic polynomial, which we will do in the sequel.

The matrix $A$ is still cyclic, so we have the inclusion $\Hilb^{reg}(\g) \subset \Hilb^{cycl}(\g)$. A representative of $\Hilb^{reg}(\g)$ is given by $(A,B=Q(A))$ where $Q$ is still an odd polynomial of degree at most $2n-1$. The idealic map is then given by 
$$I(A,B)=\langle x^{2n+1}+t_2x^{2n-1}+t_4x^{2n-3}+...+t_{2n}x, -y+\mu_2x+\mu_4x^3+...+\mu_{2n}x^{2n-1}\rangle.$$
This ideal is invariant under the map $(x,y)\mapsto (-x,-y)$. This is not surprising since a generic element of the $\mf{so}_{2n+1}$-Hilbert scheme is a pair of two diagonal matrices with entries $(x_1,...,x_n,0,-x_n,...,-x_1)$ and $(y_1,...,y_n,0,-y_n,...,-y_1)$. So they can be thought of as $2n+1$ points in $\C^2$ with one point being the origin and the other points being symmetric with respect to the origin. This set is invariant under the map $-\id$, so is its defining ideal.

The next type, $C_n$, is quite similar to $B_n$.

\subsection{Case \texorpdfstring{$C_n$}{Cn}}
 
Let $\g=\mf{sp}_{2n}$. We use the symplectic structure $\omega=\sum_i e_i\wedge e_{n+i}$ of $\C^{2n}$ to represent $\g$. So a matrix 
$$\left(\begin{array}{c|c}
  A  & B \\
\hline
  C & D
\end{array}\right)$$
is in $\g$ iff $D=-A^\top$ and $B$ and $C$ are symmetric matrices.

Fix the principal nilpotent by

$$f=\left(\begin{array}{cccc|cccc}
  &&&&&&& \\
	1&&&&&&&\\
	&\ddots&&&&&&\\
	&&1&&&&&\\ \hline
	&&&&&-1&&\\
	&&&&&&\ddots&\\
	&&&&&&&-1\\
	&&&1&&&&
\end{array}\right)$$

This element is cyclic, so its centralizer is given by odd polynomials: $Z(f)=\{\mu_2f+\mu_4f^3+...+\mu_{2n}f^{2n-1}\}$. As for $B_n$ we call the $\mu_{2i}$ higher Beltrami coefficients.

A principal slice is given by 

$$\left(\begin{array}{cccc|cccc}
  &&&& t_{2n}&&& \\
	1&&&&&t_{2n-2}&&\\
	&\ddots&&&&&\ddots&\\
	&&1&&&&&t_2\\ \hline
	&&&&&-1&&\\
	&&&&&&\ddots&\\
	&&&&&&&-1\\
	&&&1&&&&
\end{array}\right)$$

Let $A$ be an element of this form. Its characteristic polynomial is given by $x^{2n}-t_2x^{2n-2}+t_4x^{2n-4}\pm...+(-1)^nt_{2n}$. By changing signs in the $t_{2i}$ we can omit the minus signs in the characteristic polynomial.

The matrix $A$ is still cyclic so a representative of $\Hilb^{reg}(\mf{sp}_{2n})$ is given by $(A,B=Q(A))$ where $Q$ is an odd polynomial of degree at most $2n-1$.

The idealic map is given by $$I(A,B)= \langle x^{2n}+t_2x^{2n-2}+t_4x^{2n-4}+...+t_{2n}, -y+\mu_2x+\mu_4x^3+...+\mu_{2n}x^{2n-1} \rangle.$$

As for $B_n$, this ideal is invariant under $-\id$ which comes from the fact that two diagonal matrices in $\mf{sp}_{2n}$ give $2n$ points in $\C^2$ which are symmetric with respect to the origin.

The last classical type, $D_n$, has some surprises.

\subsection{Case \texorpdfstring{$D_n$}{Dn}}\label{Dn}

Let $\g=\mf{so}_{2n}$. We use the same representation as for $B_n$. 

Fix the following principal nilpotent element:
\begin{equation}
f=\left(\begin{array}{@{}cccc|cccc@{}}
&&&&&&& \\
1&&&&&&& \\
&\ddots&&&&&& \\
&&1&&&&& \\ \hline
&&1&0&&&& \\
&&&-1&-1&&& \\
&&&&&\ddots && \\
&&&&&&-1&\\
\end{array}\right).
\end{equation}

This elements is not cyclic, since $f^{2n-1}=0$. A direct computation shows that $Z(f)=\{\mu_2f+\mu_4f^3+...+\mu_{2n-2}f^{2n-3}\} \cup \{\sigma_n S\}$ where $S$ is the matrix 

\begin{equation}\label{matrixS}
S=\left(\begin{array}{@{}ccc|ccc@{}}
  &&&&& \\
	&&&&&\\
	1&&&&& \\ \hline
	-1&&&&&\\
	&&&&&\\
	&&1&-1&&
\end{array}\right)
\end{equation}

We can give an intrinsic definition of the matrix $S$: let $R$ be a root system and $v_{\alpha}$ be a root vector in $\g$ for the root $\alpha \in R$. Choose a base $\alpha_1, ..., \alpha_n$ of $R$ (the simple roots) such that $\alpha_{n-1}$ and $\alpha_n$ correspond to the two non-adjacent vertices in the Dynkin diagram of $D_n$ (see Figure \ref{Dynkin}). We can choose $f$ to be $\sum_i v_{\alpha_i}$. The matrix $S$ is then given by 
$$S=v_{\alpha_1+...+\alpha_{n-1}} \pm v_{\alpha_1+...+\alpha_{n-2}+\alpha_n}$$ where the sign depends on the choice of the root vectors.
\begin{figure}[H]
\centering
\includegraphics[height=2cm]{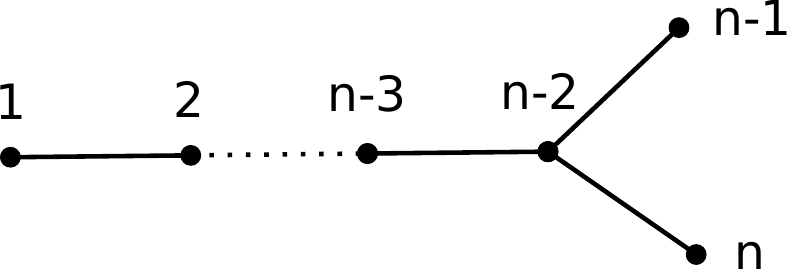}

\caption{Dynkin diagram for $D_n$}\label{Dynkin}
\end{figure}
\vspace*{0.3cm}

A representative of $\Hilb^{reg}_0(\mf{so}_{2n})$ is given by $(A=f,B=Q(f)+\sigma_nS)$ where $Q$ is an odd polynomial of degree at most $2n-3$. Such a pair is cyclic iff $\sigma_n \neq 0$. 

Let us compute the ideal in the cyclic case. One checks easily that $fS=Sf$ and that $S^2=2f^{2n-2}$. Hence for $B=\mu_2f+...+\mu_{2n-2}f^{2n-2}+\sigma_nS$, we get $AB=fB=\mu_2f^2+...+\mu_{2n-2}f^{2n-2}$ and $B^2=(\mu_2f+...+\mu_{2n-2}f^{2n-3})^2+2\sigma_n^2f^{2n-2}$. Hence, the idealic map is given by
$$I(A,B)=\langle x^{2n-1}, xy=\mu_2x^2+\mu_4x^4+...+\mu_{2n-2}x^{2n-2}, y^2=\nu_2x^2+\nu_4x^4+...+\nu_{2n-2}x^{2n-2} \rangle$$
where $\nu_{2k}=\sum_{i=1}^{k}\mu_{2i}\mu_{2k+2-2i}$ for $k=1,...,n-2$ and $\nu_{2n-2}=2\sigma_n^2+\sum_{i=1}^{n-1}\mu_{2i}\mu_{2n-2i}$. So we see that $(\mu_2, \mu_4, ..., \mu_{2n-2}, \nu_{2n-2})$ is a set of independent variables which we call higher Beltrami differentials for $D_n$. We will also call $\sigma_n$ a higher Beltrami differential.
If $\sigma_n=0$, we define the idealic map to be the continuous extension of the above ideal which is still of the same form.
\begin{Remark}
We have seen in Example \ref{idealnotcont} that inside $\Hilb^{cycl}(\g)$ there is no well-defined continuous extension of the idealic map. But inside the zero-fiber, the limit is unique.
\end{Remark}

The Hilbert scheme is covered by charts indexed by partitions (see \cite{Haiman}).
The chart in which $I$ is written corresponds to the partition $2n=(2n-1)+1$ which we write also $[2n-1,1]$. In fact, this is the highest partition of $2n$ of type $D_n$ (see \cite{Coll}, chapter 5 for special types of partitions).
%Indeed, a partition of type $D_n$ has all its even parts with even multiplicity. They nearly paramitrize nilpotent orbits of $\mf{so}_{2n}$ (some partitions give two orbits, see \cite{Coll}).

A principal slice is given by

$$\left(\begin{array}{@{}cccc|cccc@{}}
  &&&\tau_n&-\tau_n&&t_{2n-2}& \\
	1&&&&&\udots&&-t_{2n-2}\\
	&\ddots&&t_2&t_2& &\udots&\\ 
	&&1&&&-t_2&&\tau_n\\ \hline
	&&1&0&&-t_2&&-\tau_n\\
	&&&-1&-1&&&\\
	&&&&&\ddots&&\\
	&&&&&&-1&
\end{array}\right).$$

Notice that the matrix for $\tau_n$ is $S^\top$.
Let $A$ be a matrix of this type. Its characteristic polynomial is given by $$\chi(A)=x^{2n}-4t_2x^{2n-2}+4t_4x^{2n-4}\pm... +(-1)^{n-1}\times 4t_{2n-2}x^2+(-1)^n\tau_n^2.$$ By changing signs and factors in $t_{2i}$ and $\tau_n$, we can omit signs and the factor 4 in the characteristic polynomial.

One can compute that the minimal polynomial of $A$ is equal to the characteristic polynomial iff $\tau_n\neq 0$. So $A$ is cyclic iff $\tau_n\neq 0$ (by Proposition \ref{regularsln}). In that case, the centralizer consists of all odd polynomials in $A$ of degree at most $2n-1$. If $\tau_n=0$, the centralizer is given by $$Z(A)=\{\mu_2A+\mu_4A^3+...+\mu_{2n-2}A^{2n-3}\} \cup \{\sigma_n S_t\}$$ where the matrix $S_t$ is given by $S_t= S+t_{2n-2}S^\top$. 
The minimal polynomial is given by $\chi(x)/x$ (which is a polynomial since $\tau_n=0$).

The pair $(A,B)$ is cyclic iff either $\tau_n \neq 0$ or $\tau_n=0$ and $\sigma_n \neq 0$. 
In the first case, the idealic map is given by
$$I=\langle x^{2n}+t_2x^{2n-2}+t_4x^{2n-4}+...+t_{2n-2}x^2+\tau_n^2, -y+\mu_2x+\mu_4x^3+...+\mu_{2n}x^{2n-1}\rangle.$$
In the second case, we need three generators for the ideal, like for the zero-fiber. We can compute that
\begin{align*}
I(A,B)= \; \langle x^{2n-1} &= u_2x+u_4x^3+...+u_{2n-2}x^{2n-3}+uy,\\
  xy &=  v_0+v_2x^2+...+v_{2n-2}x^{2n-2}, \\
 y^2 &=  w_0+w_2x^2+...+w_{2n-2}x^{2n-2}\rangle
\end{align*}
where the coordinates can be chosen to be $(u_2,u_4,...,u_{2n-2},u,v_2,...,v_{2n-2},w_{2n-2})$, i.e. all the other variables are functions of these. For a unified way to get coordinates in Hilbert schemes, see Appendix \ref{haimancoords} or directly Haiman's paper \cite{Haiman}.

The second ideal is in the chart corresponding to the partition $[2n-1,1]$ whereas the first corresponds to the trivial partition $[2n]$. If $u \neq 0$ we can write the second ideal in the first chart, i.e. perform a coordinate change in the Hilbert scheme. The link between the coordinates is given by 
$$
\left \{ \begin{array}{cl}
\tau_n^2 = uv_0 \\
\mu_{2n}=\frac{1}{u} \\
\mu_{2k} = -\frac{u_{2k}}{u} &\text{ for } 1\leq k < n\\
t_{2k} = u_{2n-2k}+uv_{2n-2k} & \text{ for } 1\leq k \leq n-1 
\end{array}\right.
$$

A regular pair $[(A,B)]$ which is not cyclic has both $\tau_n$ and $\sigma_n$ equal to 0. In that case, we define the idealic map $I(A,B)$ to be the limit of $I(A,B+tS_t)$ for $t\in\C$ goes to 0. So we stay in a chart associated to the partition $[2n-1,1]$.

Notice that the map from $\Hilb^{reg}(\g)$ to the space of ideals $I_{\g}(\C^2)$ is not injective, since for $\tau_n$ and $-\tau_n$ we get the same ideal. Even in the zero-fiber the map is not injective, since $\sigma_n$ and $-\sigma_n$ give the same ideal. In addition, the map is not surjective neither. Indeed the ideal $I=\langle x^5-y, xy,y^2 \rangle \in I_{\g}(\C^2)$ is not in the image since with the notations above we have $v_0=0$ and $u\neq 0$. Changing the chart, we can compute that $\tau_n^2=uv_0 = 0$. But for a matrix in $\Hilb^{reg}(\g)$ with $\tau_n=0$ we get $u=0$.

\begin{Remark}
In the usual Hilbert scheme, there is only one cell of maximal dimension. Comparing type $C_n$ and type $D_n$, we see that the zero-fiber of $$\{I \text{ ideal of }\C[x,y] \mid \codim I=2n, I \text{ invariant under } -\id\}$$ has two components of maximal dimension, those corresponding to the zero-fibers $\Hilb^{reg}_0(\mf{sp}_{2n})$ and $\Hilb^{reg}_0(\mf{so}_{2n})$.
\end{Remark}

\begin{Remark}
We notice the following analogy to Higgs bundles: the pair $[(f,0)] \in \Hilb(\mf{so}_{2n})$ should correspond to the Higgs field given by $\Phi = f$ on the bundle $V=K^2\oplus K\oplus K^0 \oplus K^{-2} \oplus K^{-1}\oplus K^0$. This Higgs bundle $(V,\Phi)$ is not stable, only polystable. This could explain why the idealic map can not be continuously extended to $[(f,0)]$. The link between Higgs bundles and higher complex structures remains mysterious, see the perspectives in the introduction.
\end{Remark}

\section{\texorpdfstring{$\g$}{g}-complex structures}\label{section2}

Using the $\g$-Hilbert scheme we are able to construct a geometric structure on a smooth surface, generalizing both complex and higher complex structures. The construction and methods are inspired by those used for higher complex structures in \cite{FockThomas}. We recall the ideas of constructing higher complex structures before defining the $\g$-complex structure.

\subsection{Complex and higher complex structures}

A complex structure on a surface $\Sigma$ is completely encoded in the \textit{Beltrami differential}. 

This goes as follows: For surfaces, a complex structure is equivalent to an almost complex structure, i.e. an endomorphism $J(z)$ in $T^*_z\Sigma$ such that $J^2=-\id$ and varying smoothly with $z\in\Sigma$ ($J$ imitates the multiplication by $i$). We can diagonalize $J$ by complexifying the cotangent bundle. We get a decomposition into eigenspaces $$T^{*\C}\Sigma = T^{*(1,0)}\Sigma \oplus T^{*(0,1)}\Sigma.$$ In addition $T^{*(1,0)}\Sigma$ is the complex conjugate of $T^{*(0,1)}\Sigma$, so one determines the other. Hence, the complex structure is completely encoded in a direction in each complexified cotangent space, i.e. in a section $s$ of $\mathbb{P}(T^{*\C}\Sigma)$ which is nowhere real (meaning $s$ and $\bar{s}$ are linear independent). The projectivization can also be obtained by the zero-fiber of the punctual Hilbert scheme of length 2: $$\Hilb^2_0(\C^2) \cong \mathbb{P}(\C^2).$$ In coordinates, we can write $T^{*(0,1)}_z\Sigma = \Span(\bar{p}-\mu_2(z)p)$ where $p$ and $\bar{p}$ are linear coordinates on $T^{*\C}\Sigma$.
The coefficient $\mu_2(z)$ is the Beltrami differential. The condition that the section $s$ is nowhere real translates to $\mu_2(z)\bar{\mu}_2(z) \neq 1$ for all $z \in \Sigma$.

Generalizing this idea, we defined in \cite{FockThomas} the \textbf{higher complex structure} as a section $I$ of $\Hilb^n_0(T^{*\C}\Sigma)$ satisfying $I(z) + \bar{I}(z) = \langle p,\bar{p}\rangle$ at every point $z\in \Sigma$. Here $p$ and $\bar{p}$ are linear coordinates on $T^{*\C}\Sigma$. The condition on $I$ generalizes the condition above of a nowhere real section. We call it the \textit{non-reality constraint}.

We use exclusively the idealic viewpoint of the punctual Hilbert scheme in this definition. Since the $\g$-Hilbert scheme uses the matrix viewpoint, we have to rewrite the definition of higher complex structure in that picture.
So we replace the ideal $I(z)$ by a conjugacy class of commuting matrices $A(z)$ and $B(z)$. We can put them together in a gauge class of a $\mf{sl}_n$-valued 1-form $\Phi(z)=A(a) dz+ B(z) d\bar{z}$. The commutativity of $A$ and $B$ translates to the fact that $\Phi$ satisfies $\Phi \wedge \Phi = 0$. 

It is not surprising to use 1-forms since a generic point of the Hilbert scheme gives $n$ distinct points in each fiber $T^{*\C}_z\Sigma$ which can be put together to $n$ sections of $T^{*\C}\Sigma$, i.e. a $n$-tuple of complex 1-forms. Going to the zero-fiber of the Hilbert scheme means that all these 1-forms are collapsed to the zero-section $\Sigma \subset T^{*\C}\Sigma$.

\subsection{Definition}
We are now ready to give the definition of a $\g$-complex structure, but one difficulty stays: we have to incorporate the non-reality constraint in the matrix viewpoint. Recall the map $\mu_2: \Hilb^{reg}_0(\g) \rightarrow \C$ associating to $[(A,B)]$ the unique $\mu_2\in \C$ such that $B-\mu_2A$ is irregular (Equation \eqref{mu}). 

\begin{definition}\label{def-g-complex-1}
A \textbf{$\g$-complex structure} is a $G$-gauge class of elements locally of the form $$\Phi_1(z) dz+ \Phi_2(z) d\bar{z} \in \Omega^1(\Sigma, \g) = \Omega^1(\Sigma,\C)\otimes \g$$ such that $$[(\Phi_1(z), \Phi_2(z))] \in \Hilb^{reg}_0(\g)$$ and $\mu_2(z)\bar{\mu}_2(z) \neq 1$ for all $z\in \Sigma$.
\end{definition}

Notice that for complex structures, the map $\mu_2(z)$ is nothing but the Beltrami differential. In particular, for $\g=\mf{sl}_2$, we get a usual complex structure.

\begin{Remark}
An equivalent definition, which uses only global objects, goes as follows: A $\g$-complex structure is a pair $(V,\Phi)$ where $V$ is a trivial $G$-bundle and $\Phi\in\Omega^1(\S,\ad(V))$ satisfying
\begin{enumerate}
	\item Commutativity: $\Phi \wedge \Phi = 0$.
	\item Nilpotency: $\Phi(z).X(z)$ is nilpotent $\forall z\in \S$ and $\forall X(z)\in T_z^{*\C}\S$.
	\item Regularity: $\Phi(z).X(z)$ is regular $\forall X(z)\in T_z^{*\C}\S\backslash L(z)$ where $L(z)$ is a one-dimensional subspace of $T_z^{*\C}\S$.
	\item Non-reality: $L(z) \cap \bar{L}(z) = \{0\} \;\forall z\in \S$.
\end{enumerate}
The direction given by $L$ corresponds to $\Span(\Phi_1-\mu_2\Phi_2)$, the direction in which $\Phi$ is not principal nilpotent.
\end{Remark}

Using this line-subbundle $L$, we get the following link between $\g$-complex structure and complex structures:

\begin{prop}\label{inducedcomplex}
A $\g$-complex structure induces a complex structure on $\Sigma$.
\end{prop}
\begin{proof}
Recall the map $\mu:\Hilb^{reg}_0(\g) \rightarrow \Hilb_0(\mf{sl}_2)$ given by $\mu([(A,B)])=[(e,\mu_2 e)]$ or $[(\mu_2 e, e)]$ depending on whether $A$ or $B$ is regular (see Equation \eqref{mu}). Since a $\mf{sl}_2$-complex structure is a complex structure, the map $\mu$ induces a map from $\g$-complex structures to complex structures.
\end{proof}
\begin{Remark}
To define the map $\mu$ in \ref{mu2}, we really need $\g$ to be simple. For a semisimple (non-simple) $\g$, a $\g$-complex structure would induce several complex structures.
\end{Remark}

In the definition of a higher complex structure in \cite{FockThomas}, we use the zero-fiber $\Hilb^n_0(\C^2)$, without imposing to be in the regular part. The fact that we actually are in the regular part follows from the non-reality constraint $I+ \bar{I}=\langle p,\bar{p} \rangle$.
The same can be obtained for $\g$ of classical type, where we can reformulate the definition of $\g$-complex structures in a nicer way using the idealic map.

\subsection{Idealic viewpoint}

Recall the space of ideals $I_{\g}(\C^2)$ constructed in \ref{idealic}. Denote by $I_{\g,0}(\C^2)$ the set of those ideals of $I_{\g}(\C^2)$ which are supported at the origin (the zero-fiber). We can rewrite the definition of a $\g$-complex structure in the following way:

\begin{definition}\label{def-g-complex}
For $\g$ of classical type, a \textbf{$\g$-complex structure} is a section $I$ of $I_{\g,0}(T^{*\C}\Sigma)$ such that 
$$I(z) + \bar{I}(z)= \left \{ \begin{array}{cl}
\langle p, \bar{p} \rangle & \text{ if  } \g \text{ of type } A_n, B_n, C_n \\
\langle p, \bar{p} \rangle^2 &\text{ if } \g \text{ of type } D_n.
\end{array} \right.$$
\end{definition}

Notice that the condition on the ideals does not depend on coordinates since $\langle p, \bar{p}\rangle$ is the maximal ideal associated to the origin.

We prove the equivalence of both definitions. For that recall that to an ideal $I$ one can associate a class of commuting matrices $[(A,B)]$ (see \ref{bijhilbert}).
\begin{prop}
For classical $\g$, the condition on $I+ \bar{I}$ given in Definition \ref{def-g-complex} is equivalent to $[(A(z),B(z))]$ being in the regular part $\Hilb^{reg}_0(\g)$ and having $\mu_2\bar{\mu}_2\neq 1$, i.e. the condition in Definition \ref{def-g-complex-1}.
\end{prop}
\begin{proof}
The backwards direction is a direct computation using the preferred representatives for $\Hilb^{reg}_0(\g)$ from Proposition \ref{paramit}. So we concentrate on the direct implication.

\underline{Case $A_n$.} The case $\g$ of type $A_n$ has been treated in \cite{FockThomas}, Appendix 5.1. The idea of the proof is similar to the case $D_n$ below.

\underline{Case $B_n$.} For $\g$ of type $B_n$ the standard representation gives $\mf{so}_{2n+1} \hookrightarrow \mf{sl}_{2n+1}$. By virtue of the case $A_n$, we know that $I+ \bar{I} =\langle p, \bar{p} \rangle$ implies $\mu_2\bar{\mu}_2\neq 1$ and $(A,B)$ regular for $\mf{sl}_{2n+1}$, i.e. 
$$I(A,B)=\langle p^{2n+1}, -\bar{p}+\mu_2p+\mu_3p^2+...+\mu_{2n}p^{2n} \rangle.$$
Since we know that in case $B_n$, the ideal $I$ is invariant under the map $-\id$, we get $\mu_{2k+1}=0$ for all $k=1,...,n-1$. So $I$ corresponds to a pair $(f,Q(f))$ for $Q$ an odd polynomial of degree at most $2n-1$, which is precisely a representative of $\Hilb^{reg}_0(\mf{so}_{2n+1})$ (see Subsection \ref{Bn}).

\underline{Case $C_n$.} This case is exactly analogous to $B_n$ via the injection $\mf{sp}_{2n}\hookrightarrow \mf{sl}_{2n}$.

\underline{Case $D_n$.} We imitate the strategy of the proof for case $A_n$ in \cite{FockThomas} appendix 5.1 with only difference that we have to go further in the analysis, needing some computations. 
The main argument is an iteration process which always ends since $p^k\bar{p}^l = 0 \mod I$ for $k+l \geq 2n$.

Put $I_1 = (I \mod \langle p, \bar{p} \rangle^2)$, i.e. the set of all terms of degree at most 1 appearing in $I$.  If $I_1$ is of dimension 2, then $I=\langle p, \bar{p}\rangle$ since both $p$ and $\bar{p}$ can be expressed by higher terms which by iteration become 0. If $I_1$ is of dimension 1, then we have a relation of the form $\bar{p}=\mu_2p+p^2R(p,\bar{p})$ where $R$ is a polynomial, which gives $\bar{p}$ as a polynomial in $p$ by iteration. We can then explicitly check that $I+ \bar{I}$ is either $\langle p, \bar{p} \rangle$ or $\langle p=\bar{p}, p\bar{p}, p^2\rangle$. Hence $I_1=\{0\}.$

Put $I_2 = (I \mod \langle p, \bar{p}\rangle^3)$. We have $I_2+ \bar{I}_2 = (I+ \bar{I})_2 = \langle p^2, p\bar{p}, \bar{p}^2\rangle$ by assumption on $I$. If $I_2$ is of dimension 3, then all of $p^2, p\bar{p}$ and $\bar{p}^2$ can be expressed by higher terms. By iteration, we get $I=\langle p^2, p\bar{p}, \bar{p}^2\rangle$ which is not of type $D_n$. If $\dim I_2 \leq 1$, then we also have $\dim \bar{I}_2 \leq 1$, so $2\geq \dim I_2+\dim \bar{I}_2 = \dim \langle p^2,p\bar{p}, \bar{p}^2 \rangle_2 = 3$, a contradiction. Hence $\dim I_2=2.$

There is a term containing $p\bar{p}$ in $I_2$ since if not, no such term would neither exist in $\bar{I}_2$, so neither in $I_2+\bar{I}_2 = \langle p^2, p\bar{p}, \bar{p}^2 \rangle$, a contradiction. Without loss of generality, we can assume that there is another term containing $\bar{p}^2$ (if not change the role of $I$ and $\bar{I}$).

So there exist $\alpha, \beta, \gamma, \delta \in \C$ such that 
$$ \left \{ \begin{array}{cl}
\bar{p}^2 = \alpha p^2+\beta p\bar{p} &\mod I_2\\
p\bar{p} = \gamma p^2 +\delta \bar{p}^2  &\mod I_2
\end{array}\right. $$

\noindent If $\beta\gamma \neq 1$, we can simplify by substitution one into the other to
$$\left \{ \begin{array}{cl}
\bar{p}^2 = \alpha' p^2 &\mod I_2 \\
p\bar{p} = \gamma' p^2 &\mod I_2
\end{array}\right. $$
If $\beta\gamma = 1$, we have $p^2\in I_2$, so $p\bar{p}=\delta \bar{p}^2 \mod I_2$, so changing $I$ to $\bar{I}$ we are in the previous situation.

Iterating the substitution process we get that $\bar{p}^2$ and $p\bar{p}$ are polynomials in $p$. Using the invariance of $I$ under $-\id$, we see that these are polynomials in $p^2$, i.e. even polynomials. So the most generic ideal is given by 
$$I=\langle p^{2n-1}, p\bar{p}=\mu_2p^2+\mu_4p^4+...+\mu_{2n-2}p^{2n-2}, \bar{p}^2=\nu_2p^2+\nu_4p^4+...+\nu_{2n-2}p^{2n-2}\rangle$$
which corresponds to a regular element of $\Hilb^{reg}_0(\mf{so}_{2n})$.
One checks that $I+ \bar{I}$ with $I$ of the form above equals $\langle p,\bar{p}\rangle ^2$ iff $\mu_2\bar{\mu_2} \neq 1$.
\end{proof}

To end this section, we determine the geometric nature of the various higher Beltrami coefficients. Since $p$ and $\bar{p}$ are linear coordinates on $T^{*\C}\Sigma$, we can identify $p=\frac{\partial}{\partial z}= \partial$ and $\bar{p}=\frac{\partial}{\partial \bar{z}}=\bar{\partial}$. Denote by $K$ the canonical bundle, i.e. $K=T^{*(1,0)}\Sigma$, and by $\Gamma(B)$ the space of sections of a bundle $B$. 

Analyzing the behavior under a coordinate change $z \mapsto w(z,\bar{z})$ analogous to the computation in \cite{FockThomas} section 3.1., we get  \begin{equation}\label{naturemu}\mu_i\in \Gamma(K^{1-i}\otimes \bar{K}) \text{ and } \nu_{2i} \in \Gamma(K^{-2i}\otimes \bar{K}^2).\end{equation} 
Since $\sigma_n^2$ has the same nature as $\nu_{2n-2}$, we get $\sigma_n\in \Gamma(K^{1-n}\otimes \bar{K})$.

\section{Moduli space}\label{section3}

In this section, we define the moduli space of $\g$-complex structures and explore its properties. In most of the section $\g$ is of classical type. We first have to define an equivalence relation on $\g$-complex structures, which is accomplished by the notion of higher diffeomorphisms.

\subsection{Higher diffeomorphisms}

In order to get a finite-dimensional moduli space, it is not sufficient to quotient by the diffeomorphisms of $\Sigma$ isotopic to the identity, as for Teichm\"uller space. The reason is that the $\g$-complex structure is non-linear in the cotangent spaces $T^{*\C}_z\Sigma$. Diffeomorphisms act linearly on the cotangent space, so it cannot act much on $\g$-complex structures. 

For higher complex structures, in Section 3.2 in \cite{FockThomas} higher diffeomorphisms are defined to be Hamiltonian diffeomorphisms of $T^*\Sigma$ preserving the zero-section $\Sigma \subset T^*\Sigma$. This gives the higher diffeomorphisms for type $A_n$.
We generalize this idea to general $\g$. 

To this end, we need a faithful representation of $\g$, i.e. an injection $\rho: \g \hookrightarrow \mf{sl}_m$ for some $m \in \bb{N}^*$. This always exists by Ado's theorem. For classical $\g$, we will take the standard representation of $\g$ on $\C^m$ (i.e. $\mf{sl}_n \subset \mf{gl}_n, \mf{so}_n \subset \mf{gl}_n$ and $\mf{sp}_{2n} \subset \mf{gl}_{2n}$).

As stated several times, one should think of a $\g$-complex structure as a $m$-tuple of 1-forms with some symmetry, which collapses all to the zero-section. In a given fiber, these $m$ points are given by the common eigenvalues of the two commuting matrices. The extra symmetry expresses the fact that we deal will a subset of $\mf{sl}_m$, coming from a representation of $\g$.

The space of higher diffeomorphisms which we are looking for has to preserve this symmetry. 
To be more precise, we are interested in the eigenvalues of $\rho(g)$ for $g\in \g$. Consider the set $D_\rho$ of those $m$-tuples $(x_1, ..., x_m)\in \C^m$ which appear as the spectrum of some $\rho(g)$. By the Jordan decomposition, we can restrict attention to semisimple elements $g_s$. Since regular semisimple elements are dense, and can be conjugated to the Cartan $\h$, the set $D_\rho$ is the image $\rho(\h)$, simultaneously diagonalized in some basis of $\C^m$. Hence $D_\rho$ is a vector subspace of $\C^m$.
An $m$-tuple of points in $\C^2$ with coordinates $(x_i, y_i)_{1\leq i \leq n}$ is called \textit{$\rho(\g)$-symmetric} if both $(x_1,...,x_m)$ and $(y_1,...,y_m)$ are in $D_\rho$.

For example for $\mf{sl}_n$, we have $D_\rho = \{(x_1,...,x_n \mid \sum_i x_i = 0)\}$. So $\mf{sl}_n$-symmetric points are simply $n$ points with barycenter the origin.
That is why a higher diffeomorphisms has to preserve the zero-section. For $\g$ of type $B_n, C_n$ or $D_n$, a set of $\rho(\g)$-symmetric points is symmetric with respect to the origin.

\begin{definition}\label{g-rho-diffeo}
A \textbf{higher diffeomorphism of type $(\g, \rho)$} is a Hamiltonian diffeomorphism of $T^*\Sigma$ whose extension to $T^{*\C}\S$ preserves the space of $\rho(\g)$-symmetric $m$-tuples.
For classical $\g$, we use the standard representation and omit $\rho$. The group of all higher diffeomorphisms of type $\g$ is denoted by $\Symp(\g, \S)$.
\end{definition}

For $\g$ of type $B_n, C_n$ or $D_n$, a higher diffeomorphism is a Hamiltonian diffeomorphism of $T^*\Sigma$ invariant under the map $(z,p,\bar{p})\mapsto (z,-p,-\bar{p})$. 
In coordinates a Hamiltonian diffeomorphism is generated by a function $H(z,\bar{z},p,\bar{p})$ which can be Taylor developed to 
$\sum_{k,l} w_{k,l}(z,\bar{z})p^k\bar{p}^l$. The associated flow preserves the zero-section iff $w_{0,0}=0$. It is invariant under $-\id$ iff it has only odd terms, i.e. 
$w_{k,l}=0$ for all $k+l$ even.

\subsection{Action on \texorpdfstring{$\g$}{g}-complex structures}\label{actiondiff}

We can now analyze how higher diffeomorphisms act on $\g$-complex structures. From now on, we consider only $\g$ of classical type.

Intuitively, Hamiltonian diffeomorphisms of $T^*\Sigma$ act on the space of 1-forms, so also on $m$-tuples of them. The invariance condition implies that the symmetry of the $m$ 1-forms is preserved. This action persists at the limit when the $m$-tuple of 1-forms is collapsed to the zero-section.

To compute the action, it is better to work in the idealic viewpoint. We imitate the steps from \cite{FockThomas} section 3.2.

Let $I$ be an ideal representing a $\g$-complex structure. Write $I$ with generators $\langle f_1, ..., f_r \rangle$. Each $f_k$ can be considered as a function on $T^{*\C}\Sigma$, so its variation under a Hamiltonian $H$ is given by the Poisson bracket $\{H, f_k\}$.
The tangent space at $I$ in the space of all ideals of codimension $m$ is the set of all ring homomorphisms from $I$ to $A/I$. Thus a Hamiltonian $H$ changes $I$ to 
$\langle f_1+\varepsilon\{H,f_1\} \mod I, ..., f_r+\varepsilon \{H,f_r\} \mod I \rangle$.

We restate a lemma from \cite{FockThomas} (lemma 4) which allows to simplify $H$:
\begin{lemma}\label{simplification}
Let $I=\left\langle f_1, ..., f_r \right\rangle$ be an ideal of $\mathbb{C}[z,\bar{z},p,\bar{p}]$ such that $\{f_i, f_j\} = 0 \mod I$ for all $i$ and $j$. Then for all polynomials $H$ and all $k \in \{1,...,r\}$ we have $\{H, f_k\} \mod I = \{H \mod I, f_k\} \mod I$.
\end{lemma}
\begin{proof}
The only thing to show is that if we replace $H$ by $H+gf_l$ for some polynomial $g$ and some $l \in \{1,...,r\}$, the expression does not change. Indeed, 
$\{H+g f_l, f_k\}=\{H, f_k\}+g\{f_l,f_k\}+\{g,f_k\}f_l = \{H, f_k\}  \mod I$ using the assumption.
\end{proof}

\begin{prop}
The ideals of $\Hilb^{reg}_0(\g)$ for $\g$ classical satisfy the condition of the previous lemma.
\end{prop}
\begin{proof}
For $A_n$, we have $I=\langle p^n, \bar{p}=\mu_2p+...+\mu_np^{n-1}=Q(p) \rangle$. We compute $\{p^n, -\bar{p}+Q(p)\} = np^{n-1}\partial Q = 0 \mod I$ since there is no constant term in $Q$. 

The same argument holds for $B_n$ and $C_n$ since their ideals are special cases of the ideal of type $A_n$.

For $D_n$, the ideal $I$ is given by
\begin{align*}
\langle p^{2n-1}, p\bar{p} &= \mu_2p^2+\mu_4p^4+...+\mu_{2n-2}p^{2n-2}=Q(p)+\mu_{2n-2}p^{2n-2}, \\
 \bar{p}^2 &= \nu_2p^2+\nu_4p^4+...+\nu_{2n-2}p^{2n-2}=R(p)+\nu_{2n-2}p^{2n-2}\rangle.
\end{align*}
As before the Poisson brackets with the first generator $p^{2n-1}$ vanishes modulo $I$ since $Q$ and $R$ have no constant terms. To compute the last Poisson bracket, define $\tilde{Q}=Q/p$. By the relations in $I$, we have $R=\tilde{Q}^2+p^{2n-2}\tilde{R}$ for some polynomial $\tilde{R}$ (see Subsection \ref{Dn}). Remark further that $\{a(z, \bar{z})p^k\bar{p}^l, b(z, \bar{z})p^{k'}\bar{p}^{l'}\}=0 \mod I$ whenever $k+l+k'+l' > n-1$ since any term of degree $n-1$ in $p$ and $\bar{p}$ is in $I$ and the Poisson bracket lowers this degree by 1.
With all this, we compute
\begin{align*}
&\{-p\bar{p}+Q+\mu_{2n-2}p^{2n-2}, -\bar{p}^2+R+\nu_{2n-2}p^{2n-2}\} \\
=& \;\{-p\bar{p}+p\tilde{Q}+\mu_{2n-2}p^{2n-2}, -\bar{p}^2+\tilde{Q}^2+p^{2n-2}(\tilde{R}+\nu_{2n-2})\}\\
=& \;\{-p\bar{p}+p\tilde{Q}, -\bar{p}^2+\tilde{Q}^2\} & \text{ by degree argument}  \\ 
=& \; 2\bar{\partial}\tilde{Q}(p\bar{p}-p\tilde{Q})-2\tilde{Q}\partial \tilde{Q}(\bar{p}-\tilde{Q}) \\
=& \; 2(p\bar{p}-Q)(\bar{\partial}\tilde{Q}-\frac{\tilde{Q}}{p}\partial \tilde{Q}) \\
=& \; 2\mu_{2n-2}p^{2n-2}(\bar{\partial}\tilde{Q}-\frac{\tilde{Q}}{p}\partial \tilde{Q}) &\mod I \\
=& \; 0 &\mod I 
\end{align*}
where the last line comes from the fact that $p$ divides the polynomial $\bar{\partial}\tilde{Q}-\frac{\tilde{Q}}{p}\partial \tilde{Q}$.
\end{proof}

As a consequence, when computing the action of a Hamiltonian $H$ on a $\g$-complex structure, we can reduce it modulo $I$. In particular if $H \mod I = 0$, the higher diffeomorphism generated by $H$ does not act at all. For $\g$ of type $A_n, B_n$ or $C_n$ we can reduce $H$ to a polynomial in $p$, and for $D_n$ we can reduce it to $H=w_{-}\bar{p}+\sum_{k=0}^{n-2} w_{2k+1}p^{2k+1}$.

\subsection{Local theory}

Now, we can study the local theory of $\g$-complex structures. Let $z_0$ be a point on $\Sigma$ and take a small chart around it which sends to the unit disk $\Delta$ in the complex plane (with $z_0$ send to the origin). 
%We are able to determine the local theory only for type $A_n$, $B_n$ and $C_n$. We will give element for the type $D_n$.

\begin{thm}\label{thm1}
For $\g$ of type $A_n$, $B_n$ or $C_n$, any two $\g$-complex structures are locally equivalent under higher diffeomorphism of type $\g$.

For $\g$ of type $D_n$, all $\g$-complex structures with non-vanishing $\sigma_n$ on $\Delta$ are equivalent under higher diffeomorphisms. However, the zero locus of $\sigma_n$ on $\Delta$ is an invariant.
\end{thm}
Since we work locally, it is sufficient to show that we can send all higher Beltrami differentials to 0 using higher diffeomorphisms.
\begin{proof}
The proof for $\g$ of type $A_n$ was done in \cite{FockThomas}, Appendix 5.2, using a method in the spirit of the proof of Darboux's theorem in symplectic geometry. 

If $\g$ is of type $B_n$ or $C_n$, the standard representations realizes the $\g$-complex structure as a substructure of type $A_n$. Since the last is trivializable, so is the $\g$-complex structure in that case.

For $\g$ of type $D_n$, we use the same method as for type $A_n$ by a Hamiltonian flow argument.
We start with an ideal $I$ determined by higher Beltrami differentials $(\mu_2, \mu_4, ..., \mu_{2n-2}, \nu_{2n-2})$. The action on $\mu_{2i}$ is the same as for $\g=\mf{sl}_{2n}$ so we can trivialize them using a Hamiltonian $H$ which is a polynomial in $p$. So we are left with 
\begin{equation}\label{idealdndn}
I=\langle p^{2n-1}, p\bar{p}, -\bar{p}^2+\nu_{2n-2}p^{2n-2}\rangle.
\end{equation}
We have seen at the end of Subsection \ref{actiondiff} that in the case $D_n$, any Hamiltonian can be reduced to $H=w_{-}\bar{p}+\sum_{k=0}^{n-2}w_{2k+1}p^{2k+1}$.
The only part of this Hamiltonian acting on $\nu_{2n-2}$ is $H=w_{-}\bar{p}$, which also changes $\mu_{2n-2}$. So in order to assure that $\mu_{2n-2}$ stays zero, we use $$H=w_{-}\bar{p}+w_{2n-3}p^{2n-3}.$$

We compute the action of this Hamiltonian on the ideal $I$. For the second generator of $I$ we get:
$$\{w_{-}\bar{p}+w_{2n-3}p^{2n-3}, -p\bar{p}\} = p^{2n-2}(\bar{\partial}w_{2n-3}+\partial w_{-}\nu_{2n-2}) \mod I.$$
For the third generator of $I$ we get
$$\{w_{-}\bar{p}+w_{2n-3}p^{2n-3}, -\bar{p}^2+\nu_{2n-2}p^{2n-2}\} = p^{2n-2}(w_{-}\bar{\partial}\nu_{2n-2}+2\bar{\partial}w_{-}\nu_{2n-2}) \mod I.$$
Denote by $\mu_{2-2}^t$ and $\nu_{2n-2}^t$ the image of $\mu_{2n-2}$ and $\nu_{2n-2}$ under the flow generated by $H$ at time $t$.
From the above computation we get 
$$ \left \{\begin{array}{cl}
\frac{d}{dt} \mu_{2n-2}^t &= \bar{\partial}w_{2n-3}+\partial w_{-}\nu_{2n-2} \\
\frac{d}{dt} \nu_{2n-2}^t &= (w_{-}\bar{\partial}+2\bar{\partial}w_{-})\nu_{2n-2}
\end{array}\right. $$

Instead of keeping $\nu_{2n-2}$, we work with the higher Beltrami differential $\sigma_n$. Since all the $\mu_{2i}$ are zero in $I$, we have $\nu_{2n-2}=\sigma_n^2$. Therefore we get from the second equation above $\frac{d}{dt}(\sigma_n^2)=(w_{-}\bar{\partial}+2\bar{\partial}w_{-})(\sigma_n^2)$ which gives
\begin{equation}\label{varsigma}
\frac{d}{dt}\sigma_n^t = \bar{\partial}(w_{-}^t\sigma_n^t).
\end{equation}

We wish to have $\frac{d}{dt} \mu_{2n-2}^t=0$ to stay with $\mu_{2n-2}=0$. For $\sigma_n$, we show that we can deform it to the constant function 1 on the unit disk, assuming $\sigma_n$ vanishes nowhere on $\Delta$. We choose the path $\sigma_n^t=(1-t)\sigma_n^0+t$ from the initial $\sigma_n^0$ to the constant function 1. If $\sigma_n^t=0$ for some $t$, we have to modify slightly the path. We get $\frac{d}{dt}\sigma_n^t=1-\sigma_n^0$. 

Denote by $T$ the local inverse of the $\bar{\partial}$-operator, i.e. $\bar{\partial}(Tf)=f=T\bar{\partial}f$ for all $f\in L^2(\Delta)$. The operator $T$ is a pseudo-differential operator given by
$$Tf(z)=\frac{1}{2\pi i}\int_{\mathbb{C}}\frac{f(\zeta)}{\zeta-z}d\zeta \wedge d\bar{\zeta}.$$

We can solve Equation \eqref{varsigma} with $T$: $$w_{-}^t=\frac{1}{\sigma_n^t}T(1-\sigma_n^0).$$
Putting this solution into the equation for $\frac{d}{dt}\mu_{2n-2}^t$, we can solve for $w_{2n-3}$:
$$w_{2n-3}^t=-T(\partial w_{-}^t\nu_{2n-2}^t).$$

Finally, we multiply $H$ by a bump function, a function on $\Delta$ which is 1 in a neighborhood of the origin and 0 outside a bigger neighborhood of the origin, which ensures that the Hamiltonian vector field is compactly supported, so it can be integrated to all times. In particular for $t=1$ we get $\sigma_n(z)=1$ for all $z$ near the origin.

To show that the zero locus of $\sigma_n$ can not be changed by a higher diffeomorphism, consider the singularity defined by $$f(p, \bar{p})=-\frac{\nu_{2n-2}}{2n-1}p^{2n-1}+p\bar{p}^2=0$$ which is a Kleinian singularity of type $D_{2n}$ if $\nu_{2n-2}\neq 0$. Its deformation ideal $\left\langle \frac{\del f}{\del p}, \frac{\del f}{\del \bar{p}}  \right\rangle$ is directly linked to our ideal $I$ from Equation \eqref{idealdndn} by $$\left\langle \frac{\del f}{\del p}, \frac{\del f}{\del \bar{p}}  \right\rangle + \langle p, \bar{p}\rangle^{n-1} = \langle p^{2n-1}, p\bar{p}, -\bar{p}^2+\nu_{2n-2}p^{2n-2} \rangle.$$ 
Since the type of a singularity is invariant under diffeomorphisms, so is its deformation ideal. This is why we cannot change $\nu_{2n-2}=0$ to $\nu_{2n-2}\neq 0$ by higher diffeomorphisms. 
\end{proof}
\begin{Remark}
It is interesting to notice the appearance of Kleinian singularities, which have an $ADE$-classification. The fact that for $\g$ of type $D_n$ the singularity is of type $D_{2n}$ is linked to the representation of $\mf{so}_{2n}$ on $\C^{2n}$. There should be a more intrinsic way to link $\g$-complex structures to singularities of type $\g$.

An idea in this direction is the following: the singularity of type $\g$ appears inside the Lie algebra $\g$, more precisely inside the nilpotent variety along the subregular locus (see \cite{Steinberg}). A minimal resolution of this singularity is given by the Springer resolution. There should be a link between $\g$-Hilbert schemes and the Springer resolution.
\end{Remark}

Since there are no local invariants for $\g$-complex structures, only their global geometry is non-trivial.

\subsection{Definition of the moduli space}

To define the moduli space of $\g$-complex structures, there is one more subtlety: in order to get one component, we have to fix an orientation on $\Sigma$. We then call a complex structure \textit{compatible} if the induced orientation coincides with the given orientation on $\Sigma$. We call a $\g$-complex structure \textit{compatible} if the induced complex structure is.
\begin{definition}
The moduli space $\hat{\mc{T}}_{\g}$ is the space of all compatible $\g$-complex structures modulo the action of higher diffeomorphisms of type $\g$.
\end{definition}

Notice that a $\g$-complex structure is compatible iff $\mu_2(z)\bar{\mu}_2(z) < 1$. 
Reverting the orientation on $\Sigma$ we get another copy of $\hat{\mc{T}}_{\g}$ corresponding to those $\g$-complex structures with $\mu_2(z)\bar{\mu}_2(z) > 1$. 

\begin{Remark}
One might define a moduli space of $\g$-complex structures for general $\g$ (not of classical type), by using a representation $\rho:\g \hookrightarrow \mf{sl}_m$. We conjecture that the associated moduli space does not depend on the choice of the representation $\rho$.
\end{Remark}

For $\g=\mf{sl}_2$ we get Teichm\"uller space since we can reduce any Hamiltonian to $H=w(z,\bar{z})p$ which generates a linear diffeomorphism of $T^*\Sigma$, coming from a diffeomorphism on $\Sigma$ isotopic to the identity.

For general $\g$, there is a copy of Teichmüller space inside:
\begin{prop}\label{teich-copy}
There is an injective map from Teichm\"uller space into the moduli space $\hat{\mc{T}}_{\g}$.
\end{prop}
\begin{proof}
The proposition follows from the map $\psi:\Hilb(\mf{sl}_2) \rightarrow \Hilb^{reg}(\g)$ constructed in Equation \eqref{teichcopy} in Section \ref{mu2}. This map restricts to a map between the zero-fibers and extends over the surface $\Sigma$. Finally the map descends to the quotient by higher diffeomorphisms since for $\mf{sl}_2$ we only quotient by diffeomorphisms of $\Sigma$.
In terms of higher Beltrami differentials, this map is simply given by $\psi([\mu_2]) = [(\mu_2,0,...,0)].$

For injectivity, suppose $[(\mu_2,0,...,0)]$ is equivalent to $[(\mu_2',0,...,0)]$ via a higher diffeomorphism generated by $H$. Since terms of degree 2 or more do not affect $\mu_2$, the equivalence is already obtained by the linear part of $H$, which is the extension of a diffeomorphism of $\S$. This diffeomorphism of $\S$ sends $\mu_2$ to $\mu_2'$, so they are equivalent.
\end{proof}

Furthermore, the moduli space has the following properties:
\begin{thm}\label{thm2}
For $\g$ of type $A_n, B_n$ or $C_n$, and a surface $\Sigma$ of genus $g\geq 2$, the moduli space $\hat{\mathcal{T}}_{\g}$ is a contractible manifold of complex dimension $(g-1)\dim \g$. 
Further, along the copy of Teichmüller space from Proposition \ref{teich-copy}, the cotangent bundle at any point $I$ is given by 
$$T^*_{I}\hat{\mathcal{T}}_{\g} = \bigoplus_{m=1}^{r} H^0(K^{m_i+1})$$ where $(m_1,...,m_r)$ are the exponents of $\g$ and $r=\rk \g$ denotes the rank of $\g$.

\noindent For type $D_n$, the moduli space $\hat{\mathcal{T}}_{\g}$ is a contractible topological space. The locus where the zero-set of the higher Beltrami differential $\sigma_n$ is a discrete set on $\S$ is a smooth manifold with the same properties as above (dimension, cotangent space), with the only difference that we have to take another copy of Teichmüller space (not the one from Proposition \ref{teich-copy}).
\end{thm}

Notice that the differentials in $H^0(K^{m_i+1})$ are holomorphic with respect to the complex structure induced from the $\g$-complex structure (see Proposition \ref{inducedcomplex}).

For the case $D_n$, we conjecture that the moduli space $\hat{\mathcal{T}}_{\g}$ is a topological manifold everywhere. The points where the zero-set of $\sigma_n$ is not discrete can have a cotangent space which is strictly bigger than the space of holomorphic differentials. One can think for example of the curve in $\R^2$ given by $t\mapsto (t^3,t^2)$, shown in Figure \ref{curvett2}, which has a cusp at the origin, but is still a topological manifold.

\vspace*{0.3cm}
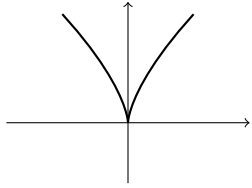
\begin{figure}[h!]
\begin{center}
\begin{tikzpicture}[scale=4]
	\draw[->, thin] (0,-0.2)--(0,0.4);
	\draw[->, thin] (-0.4,0)--(0.4,0);
\draw [domain=-0.6:0.6, thick] plot (\x*\x*\x, \x*\x);
\end{tikzpicture}
\caption{Curve with cusp}
\label{curvett2}
\end{center}
\end{figure}

\begin{proof}
The case for $A_n$ has been treated in Theorem 2 of \cite{FockThomas}. The cases $B_n$ and $C_n$ are exactly analogous:

One shows that at every point, the cotangent space exists. From this follows that $\hat{\mc{T}}_{\g}$ is a manifold.
We have to check the appearance of the exponents of the Lie algebra. Since $\mu_{2i}$ is a section of $K^{1-2i}\otimes \bar{K}$ (see Equation \eqref{naturemu}) its dual $t_{2i}$ is a section of $K^{2i}$. Since the exponents for $B_n$ and $C_n$ are the same and equal to $(1,3,...,2n-1)$, we get the desired form stated in the theorem.

For $\g$ of type $D_n$ we consider the subset on $\hat{\mathcal{T}}_{\g}$ where the zero-locus of $\sigma_n$ is discrete in $\S$. Note that the copy of Teichmüller space given by Proposition \ref{teich-copy} has $\sigma_n(z)=0$ for all $z\in \S$. Instead, we consider an injection $\T^2 \hookrightarrow \hat{\mc{T}}_\g$ given by $\mu_{2i}(z)=0$ for all $i=2,...,n-1$ and $\sigma_n(z)$ a fixed smooth section vanishing on a finite number of points.

Along this part, we know that the variation of $\mu_{2i}$ under a higher diffeomorphism generated by $H=w_{-}\bar{p}+\sum_{k=0}^{n-2}w_{2k+1}p^{2k+1}$ is given by $\delta \mu_{2i}=\bar{\partial}w_{2i-1}$ and Equation \eqref{varsigma} gives $\delta \sigma_n=\bar{\partial}(w_{-}\sigma_n)$. The variation of $\mu_{2i}$ is the same as in the case of type $A_n$, so we know that these contribute to the cotangent bundle by a term $H^0(K^{2i})$. For the term $\sigma_n$ we use the pairing between differential of type $(1-n,1)$ and of type $(n,0)$ given by integration over the surface. We get 
\begin{align*}
(\{\delta\sigma_n\} / \bar{\partial}(w_{-}\sigma_n))^* &=  \{t_n \in \Gamma(K^n) \mid \textstyle\int t_n \bar{\partial}(w_{-}\sigma_n) = 0 \; \forall \, w_{-} \in \Gamma(\bar{K}) \} \\
&=  \{t_n \in \Gamma(K^n) \mid \textstyle\int \bar{\partial}t_n w_{-}\sigma_n = 0 \; \forall \, w_{-} \in \Gamma(\bar{K}) \} \\
&= \{t_n \in \Gamma(K^n) \mid \bar{\partial}t_n = 0 \} \\
&= H^0(K^n)
\end{align*}
where we used that $\sigma_n$ vanishes only on a discrete set.

Hence the cotangent bundle is given by $$T^*_I\hat{\mathcal{T}}_{\g} = \bigoplus_{m=1}^{n-1}H^0(K^{2m})\oplus H^0(K^n).$$
The exponents of $\mf{so}_{2n}$ are precisely $(1,3,...,2n-3,n-1)$, so the cotangent bundle is of the form stated in the theorem.

For the dimension of $\hat{\mathcal{T}}_{\g}$, we use $\dim H^0(K^{m_i+1})=(g-1)(2m_i+1)$ by Riemann-Roch (using $g \geq 2$). We get $$\dim \hat{\mc{T}}_{\g}=(g-1)\sum_{i=1}^r (2m_i+1)=(g-1)\dim \g$$ using a well-known formula coming from the decomposition of $\g$ as $\mf{sl}_2$-module using the principal $\mf{sl}_2$-triple.

Contractibility for all types is analogous to the case $A_n$.
\end{proof}

From the previous theorem, we see that our moduli space $\hat{\mc{T}}_{\g}$ shares a lot of properties with the $G$-Hitchin component, in particular the dimension, contractibility and the copy of Teichmüller space.
For $G$-Hitchin components, this copy of Teichmüller space can be described as follows: any representation of the $G$-Hitchin component is a deformation of a representation of the form $$\pi_1(\Sigma) \rightarrow PSL_2(\R) \rightarrow G$$ where the first map is a Fuchsian representation and the second one is the principal map. These maps form a copy of $\T^2$. Note that in both situations, Hitchin component and $\hat{\mc{T}}_\g$, the copy of Teichmüller space is constructed using the principal map.

Of course, we conjecture the equivalence of Hitchin's component and the moduli space of $\g$-complex structures:
\begin{conj}\label{conjmaj}
The moduli space $\hat{\mathcal{T}}_{\g}\Sigma$ is canonically homeomorphic to Hitchin's component in the character variety $\Hom(\pi_1(\Sigma),G)/G$ where $G$ is the real split Lie group associated to $\g$.
\end{conj}

\subsection{Spectral curve}\label{spectralcurve}

In this part, we construct a spectral curve in $T^{*\C}\Sigma$ associated to a cotangent vector to $\hat{\mathcal{T}}_{\g}$, i.e. a $\g$-complex structure and a set of holomorphic differentials. 

The case for $\g$ of type $A_n$ was treated in \cite{FockThomas}, Section 4. In that paper, we proved that the zero-fiber $\Hilb^n_0(\C^2)$ is Lagrangian in the reduced Hilbert scheme $\Hilb^n_{red}(\C^2)$. This stays true for all classical $\g$:
\begin{prop}
The regular zero-fiber $\Hilb^{reg}_0(\g)$ is a Lagrangian subspace of $\Hilb^{reg}(\g)$ for classical $\g$.
\end{prop}
\begin{proof}
Since we are in the regular part, Proposition \ref{paramit} gives a parametrization. For classical $\g$, via the standard representation we can consider $\Hilb^{reg}(\g)$ as subset of $\Hilb^m_{red}(\C^2)$ which remains symplectic and we can explicitly check that the zero-fiber $\Hilb^{reg}_0(\g)$ is Lagrangian.
\end{proof}

For general $\g$, we conjecture the following, based on Conjecture \ref{git-conj}:
\begin{conj}
The conjectural smooth version of the $\g$-Hilbert scheme is symplectic and the zero-fiber is a Lagrangian subspace.
\end{conj}
If we assume Conjecture \ref{conj1} true, stating that the modified version of the $\g$-Hilbert scheme is a minimal resolution of $\h^2/W$, we get a symplectic structure. Indeed $\h^2=T^*\h$ has a canonical symplectic structure, which is invariant under the action of $W$. Hence it lifts to the minimal resolution.

\bigskip
Now we construct the spectral curve. First, we look at $\g$ of type $A_n$, $B_n$ or $C_n$.
We can write a cotangent vector in $T^*\hat{\mc{T}}_{\g}$ as an equivalence class of higher Beltrami differentials $\mu_i$ and holomorphic differentials $t_i$. To write in a uniform way, set $\mu_i$ or $t_i$ to 0 whenever it does not appear for $\g$. For example for type $B_n$ or $C_n$ all variables with odd index are 0.

Associate polynomials $P(p)=p^m+\sum_i t_ip^{m-i}$ and $Q(p, \bar{p})=-\bar{p}+\sum_i \mu_ip^{i-1}$ (where $m$ is the dimension of the standard representation of $\g$). Put $I=\langle P, Q \rangle$. Define the \textbf{spectral curve} $\tilde{\Sigma} \subset T^{*\C}\Sigma$ by the zero set of $P$ and $Q$. It is a ramified cover over $\Sigma$ with $m$ sheets. 

For $\g$ of type $D_n$, a generic point in the cotangent bundle $T^*\hat{\mc{T}}_{\g}$ corresponds to the ideal 
$$I=\langle p^{2n}+t_2p^{2n-2}+...+t_{2n-2}p^2+\tau_n^2, -\bar{p}+\mu_2p+...+\mu_{2n}p^{2n-1} \rangle$$ which can be seen as a special case of $A_n$. Thus we can proceed as above. In the case where $\tau_n = 0$ we have seen in \ref{Dn} that the ideal changes to an ideal with three generators. The zero-set of these generators still define a spectral curve in $T^{*\C}\Sigma$. It is the limit of the curve when $\tau_n \rightarrow 0$.

\begin{prop}
The spectral curve $\tilde{\Sigma}$ is Lagrangian to order 1 in the holomorphic differentials $t$.
\end{prop}
This is the precise analogue of Proposition 5 in \cite{FockThomas}. 
\begin{proof}
In the case where the ideal has two generators $P$ and $Q$ this is equivalent to $\{P,Q\} = 0 \mod I \mod t^2$ for $I\in T^*\hat{\mc{T}}_{\g}$.
For $A_n$, the proof is given in \textit{loc. cit}. For $B_n$ and $C_n$ it is completely analogous since the $\g$-complex structure can be seen as a special case of $A_n$.

For $\g$ of type $D_n$, a generic ideal has still two generators, so we have a special case of $A_n$. If the ideal has three generators, the spectral curve is still Lagrangian since it can be obtained as a limit of Lagrangian curves, and the property of being Lagrangian is closed.
\end{proof}

Since the spectral curve is Lagrangian to order 1, the periods are well-defined up to this order. The ratios of these periods should give coordinates on $T^*\hat{\mc{T}}_{\g}$ and also on $\hat{\mc{T}}_{\g}$.
For the trivial $\g$-complex structure (where all higher Beltrami differentials are 0) we recover Hitchin's spectral curve.

Finally, we can recover the same spectral data as Hitchin in his paper on stable bundles \cite{Hit3}. 
From a $\g$-complex structure we get a bundle $V$ over the surface $\Sigma$ whose fiber at a point $z\in \Sigma$ is $\C[p,\bar{p}]/I(z)$ where we use the idealic viewpoint. We also get a line bundle $L$ on $\tilde{\Sigma}$ whose fiber is the eigenspace of $M_p$, the multiplication operator by $p$ in the quotient $\C[p,\bar{p}]/I$. This gives the spectral data for type $A_n$.

For $\g$ of type $C_n$, we get in addition an involution $\sigma$ on the spectral curve $\tilde{\Sigma}$ given by $(p,\bar{p})\mapsto (-p, -\bar{p})$.
For $\g$ of type $D_n$, the spectral curve is singular, having a double point. The spectral data is given by a desingularization of $\tilde{\Sigma}$, the involution $\sigma$ as for $C_n$ and the line bundle $L$.
For $\g$ of type $B_n$, there is a canonical subbundle $V_0 \subset V=\C[p,\bar{p}]/I$ generated by the span of the image of $1\in \C[p, \bar{p}]$ in the quotient $\C[p,\bar{p}]/I$ (since for $B_n$, we have $I\subset \langle p, \bar{p}\rangle$). Thus the vector bundle $V$ is an extension $V_0 \rightarrow V\rightarrow V_1$. The spectral data is given by $(V_0, V_1, \sigma, L, \tilde{\Sigma})$.

\section{Higher complex structures for real Lie algebras}\label{realcase}

\subsection{Motivation and preliminaries}

Hitchin's approach to character varieties proceeds in two steps: he considers Higgs bundles for a complex group $G_{\C}$, which by the non-abelian Hodge correspondence describe the complex character variety $\Rep(\pi_1\S, G_{\C})$, and then he finds a subset, invariant under an involution, which corresponds to the representations in the split real group.

There is a notion of Higgs bundles, associated to some real Lie group $G_{\R}$, whose representations have values in $G_\R$. For the definition of these $G_\R$-Higgs bundles (or just $G$-Higgs bundles), see \cite{HKR-section}, Section 5.

We want to define a counterpart of $G_\R$-Higgs bundles in our language of punctual Hilbert schemes. To do this, we need a generalisation of Kostant's theory of regular elements and principal slices to a real Lie algebra. This was done in the paper of Kostant and Rallis \cite{KR}. We give a short summary of the material we need in this section.

Consider a simple real Lie algebra $\g_\R$. Fix a \textit{Cartan decomposition} $$\g_\R = \mf{k}_\R\oplus \mf{p}_\R.$$ 
All elements of $\mf{p}_\R$ are semisimple, so there are no nilpotent elements.
This is why we pass to the complexifications $\mf{p}_\C$ and $\mf{k}_\C$. \emph{Whenever we speak about a complex object, we might omit the index, so we will write $\mf{p}$ instead of $\mf{p}_\C$ etc.}
%The theory of these complexifications mirrors the theory for $\g_\C$.

Denote by $\mf{a}_\R$ a maximal abelian subalgebra of $\mf{p}_\R$. The dimension of $\mf{a}_\R$ is called the \textit{real rank} of $\g_\R$. Its complexification $\mf{a}$ is called the \textit{``baby Cartan''}. For an element $x\in \mf{p}$, it can be shown that 
$$\dim Z_{\mf{p}}(x) \geq \dim \mf{a}$$
where $Z_{\mf{p}}(x)$ denotes the centralizer of $x$ inside $\mf{p}$.
Elements for which equality holds are called \textit{regular}. 

Denote by $\theta$ the \textit{Cartan involution}, defined by $\theta = \id$ on $\mf{k}_\R$ and $\theta = -\id$ on $\mf{p}_\R$. It extends to a Lie algebra involution on the complexification $\g$ and on the group $G$.
A central role is played by $K_\theta := \{g \in G \mid \theta(g) = g\}$. It clearly contains $K = \exp (\mf{k})$ but is strictly bigger (see Proposition 1 in \cite{KR}).

The philosophy of the Kostant-Rallis can then be summarized by: \textit{the analogue of the Kostant theory for $\g$ in the ``real'' case is obtained by replacing $\g$ by $\mf{p}$ and $G$ by $K_\theta$}. Note that the objects we manipulate are complex, but come from a real form.

Let us give three examples of this philosophy. First, the analogue of $\g^{reg}/G \cong \h / W$ for real Lie algebras reads (see Theorem 12 in \cite{KR} for the version on the polynomial function level) 
\begin{equation}\label{preg-conj}
\mf{p}^{reg}/K_\theta \cong \mf{a} / W(\mf{a}).
\end{equation}
Second, there is a unique open dense $K_\theta$-orbit in the nilpotent variety of $\mf{p}$ (see Theorem 6 in \cite{KR}).
This is the precise analogue of the principal nilpotent orbit for complex Lie algebras.
Third, the nilpotent orbits of $\g_\R$ are in one-to-one correspondence to the nilpotent $K$-orbits in $\mf{p}$. This is the so-called \textit{Kostant-Sekiguchi correspondence}.

\subsection{Hilbert scheme associated to real Lie algebras}

Following the philosophy of the Kostant-Rallis paper, we define a Hilbert scheme associated to a real Lie algebra $\g_\R$ by imitating Definition \ref{maindef}. 

\begin{definition}
The punctual Hilbert scheme associated to a real simple Lie algebra $\g_\R$ is defined by 
$$\Hilb(\g_\R) = \{(A,B)\in \mf{p}^2 \mid [A,B]=0, \dim Z_{\mf{p}}(A,B)=\rk \g_\R\} / K_\theta$$
where $Z_{\mf{p}}(A,B)$ is the common centralizer in $\mf{p}$ and $K_\theta=\{g \in G \mid \theta(g)=g\}$.

The zero-fiber are those pairs $(A,B)$ which are nilpotent.
\end{definition}

\begin{Remark}
Conceptually, it might be better to consider pairs $(\g,\theta)$ of a complex Lie algebra $\g$ and a holomorphic involution $\theta$. We get the setting for $\g_\R$ by taking the Cartan involution, and we get the setting for $\g$ by considering $\g\times\g$ and $\theta(x,y) = (y,x)$.
Hence, this puts both situations into the same framework and emphasizes that all objects are holomorphic.
\end{Remark}

Let us analyze the example of the split real form $\mf{sl}_2(\R)$. We will see the necessity of using $K_\theta$, and not only $K = \exp(\mf{k})$.

\begin{example}
Consider $\g_\R = \mf{sl}_2(\R)$. The Cartan decomposition is given by 
$$\mf{sl}_2(\R) = \mf{so}(2) \oplus \mf{p}_\R$$
where $\mf{p}_\R = \{A \in\mf{sl}_2(\R) \mid A^T = A\}$ is the set of symmetric matrices.
Thus, a matrix in $\mf{p}$ is given by 
\begin{equation}\label{mC}
\begin{pmatrix} a & b \\ b & -a\end{pmatrix}
\end{equation}
for $a, b \in \C$. Further, we have $K = \SO(2,\C)$ and $K_\theta = \Ortho(2,\C)$.

A direct computation gives that two matrices $A$ and $B$ in $\mf{p}$ commute iff $B= \mu A$ for some $\mu \in \C P^1$ ($\mu = \infty$ means that $(A,B)=(0,B)$).

In the zero-fiber, we have nilpotent matrices in $\mf{p}$, so we have $a^2+b^2 = 0$ using the parametrization from Equation \eqref{mC}. Hence we have $b=\pm ia$ and we get two possibilities: 
\begin{equation} \label{nilpsl2}
a\begin{pmatrix} 1 & i \\ i & -1\end{pmatrix} \;\; \text{ or }\;\; a\begin{pmatrix} 1 & -i \\ -i & -1\end{pmatrix}.
\end{equation}
Let us first compute the action of $K = \SO(2,\C)$. An element of $K$ is of the form $\left(\begin{smallmatrix} \cos t & \sin t \\ -\sin t & \cos t\end{smallmatrix}\right)$ where $t$ is a complex parameter. One computes that the action of this element on a nilpotent matrix is given by multiplication by $\cos(2t)+i\sin(2t)$ which can be any non-zero complex number (recall that $t\in \C$). Therefore, in the list \eqref{nilpsl2} we can choose $a=1$ using the $\SO(2,\C)$-action.

We see that there are two nilpotent $K$-orbits which would give two components in the zero-fiber of the punctual Hilbert scheme for $\mf{sl}_2(\R)$. 
Using the conjugation by $K_\theta = \Ortho(2,\C)$, the two matrices from Equation \eqref{nilpsl2} are $\Ortho(2,\C)$-conjugated to each other. So we get only one principal nilpotent $K_\theta$-orbit.
\end{example}

Let us analyze some properties of the Hilbert scheme $\Hilb(\g_\R)$.
First, there is a natural map 
\begin{equation}\label{grgc}
\Hilb(\g_\R) \rightarrow \Hilb(\g)
\end{equation}
coming from the inclusions $\mf{p} \subset \g$ and $K_\theta \subset G$. This map is injective since $G$-conjugated points in $\mf{p}$ are $K_\theta$-conjugated.

Using this inclusion, we can define an \textit{idealic map} by composition $\Hilb(\g_\R) \rightarrow \Hilb(\g)\rightarrow I_\g$.
For $\mf{sl}_2(\R)$, this gives as one might expect $$\left(\left(\begin{smallmatrix}1 & i \\ i & -i\end{smallmatrix}\right), \mu \left(\begin{smallmatrix}1 & i \\ i & -i\end{smallmatrix}\right)\right) \mapsto \left\langle x^2, y-\mu x\right\rangle.$$

In analogy with the complex case, we conjecture that modulo some identifications of points, $\Hilb(\g_\R)$ is a resolution of $(\mf{a} \times \mf{a}) / W(\mf{a})$ and is covered by charts associated to the nilpotent orbits of $\g_\R$.
Note that by the Kostant-Sekiguchi correspondence, the nilpotent $K$-orbits in $\mf{p}$ are in bijection to the nilpotent orbits in $\g_\R$. Thus, the second statement generalizes Conjecture \ref{charts-by-nilpotent} to the real case.

In the case of a split real form, our Hilbert scheme gives in fact nothing new:
\begin{thm}\label{split-complex}
For the split real form $\g_{split}$, we have $$\Hilb^{reg}(\g_{split}) \cong \Hilb^{reg}(\g).$$
\end{thm}

\begin{proof}
The split form has one key property which makes the link to $\g$: the Cartan subalgebra $\h$ of $\g$ can be chosen to be the baby Cartan $\mf{a}\subset \mf{p}$. In particular, the real rank of $\g_{split}$ is the same as the rank of $\g$. 

Consider the open dense part of $\Hilb^{reg}(\g)$ where the first element in a pair $[(A,B)]$ is regular. Then we can quotient out the $G$-conjugation action on the first element to get 
$$\Hilb^{reg}(\g) \cong \g^{reg}/G \times \C^{\rk \g}$$
since $Z(A) \cong \C^{\rk \g}$.
The same argument applied to $\g_{split}$ yields
$$\Hilb^{reg}(\g_{split}) \cong \mf{p} / K_\theta \times \C^{\rk \g_{split}}.$$
Now, we have $$\g^{reg}/G \cong \h / W = \mf{a} /W(\mf{a}) \cong \mf{p} / K_\theta$$
where we used $\mf{a} = \h$, and $\rk \g = \rk \g_{split}$.

Therefore we get a bijection between the two Hilbert schemes.
\end{proof}

Let us analyze the functorial behavior for $\Hilb(\g_\R)$. The two natural maps from Subsection \ref{mu2} generalize.
In \cite{KR}, Theorem 11, Kostant-Rallis prove the existence of a principal $\mf{sl}_2$-triple $(e,f,h)$ in $\mf{p}$ (unique up to $K_\theta$-action) and provide a principal slice given by $e+Z(f)$. Hence, we can imitate exactly the argument which gave Equation \eqref{teichcopy} to get a map $$\Hilb^{reg}(\mf{sl}_2) \rightarrow \Hilb^{reg}(\g_\R).$$

As in the case for complex $\g$, there is a map $\mu_2: \Hilb^{reg}_0(\g_\R) \rightarrow \C$ which to $[(A,B)]$ associates the unique complex number $\mu_2$ such that $A-\mu_2B$ is not principal nilpotent. The proof is completely analogous to the one of Proposition \ref{mu2prop} using \cite{KR}, Theorem 5.

\subsection{$\g_\R$-complex structures}\label{grcomplexsection}

Once we have the notion of a punctual Hilbert scheme associated to $\g_\R$, it is straight forward to define a $\g_\R$-complex structure, imitating Definition \ref{def-g-complex-1}:
\begin{definition}
A \textbf{$\g_\R$-complex structure} on $\S$ is a $K$-gauge class of matrix-valued 1-forms which locally can be written as $\Phi_1(z)dz+\Phi_2(z)d\bar{z} \in \Omega^1(\S, \mf{p})$, such that $[(\Phi_1(z),\Phi_2(z))]\in \Hilb^{reg}_0(\g_\R)$ and $\mu_2\bar{\mu}_2 \neq 1$.
\end{definition}

From the inclusion $\Hilb(\g_\R) \hookrightarrow \Hilb(\g)$ we get:
\begin{prop}\label{grginduction}
 A $\g_\R$-complex structure induces a $\g$-complex structure. In particular it induces a complex structure.
\end{prop}
\begin{proof}
The inclusion $\Hilb(\g_\R) \rightarrow \Hilb(\g)$ (see Equation \ref{grgc}) stays true on the level of the regular part of the zero-fiber. Hence, by the definition, a $\g_\R$-complex structure induces a $\g$-complex structure. The rest follows from Proposition \ref{inducedcomplex}.
\end{proof}

Next, we wish to define higher diffeomorphisms and the moduli space. As for $\g$, we have to use an embedding of $\g_\R$ into some $\mf{sl}_m(\R)$ which always exists by Ado's theorem.

Fix some inclusion $\g_\R \hookrightarrow \mf{sl}_m(\R)$ inducing an inclusion $\mf{p} \hookrightarrow \mf{sl}_m(\C)$. By simultaneous diagonalization inside $\mf{sl}_m(\C)$, a pair $[(A,B)]\in \Hilb^{reg}(\g_\R)$ gives $m$ points in $\C^2$ (the coordinates of these points being the eigenvalues of $(A,B)$). These $m$ points satisfy some symmetry property, expressing the fact that they come from $\Hilb(\g_\R)$.

In analogy with Definition \ref{g-rho-diffeo}, we define a \textbf{higher diffeomorphism of type $(\g_\R, \rho)$} as a Hamiltonian diffeomorphism of $T^*\S$ whose extension to $T^{*\C}\S$ preserves this symmetry property. For $\g_\R$ a real form of a classical Lie algebra $\g$, we use the standard inclusion. We restrict to classical types from now on.

Finally, we define the \textbf{moduli space of $\g_\R$-complex structures}, denoted by $\hat{\mc{T}}_{\g_\R}$, as the equivalence classes of $\g_\R$-complex structures
under the action of higher diffeomorphisms of type $\g_\R$.

Since a $\g_\R$-complex structure induces a complex structure (see Proposition \ref{grginduction}), we get a map from $\hat{\mc{T}}_{\g_\R}$ to Teichmüller space.
Using the map $\Hilb^{reg}(\mf{sl}_2) \rightarrow \Hilb^{reg}(\g_\R)$ (see paragraph before \ref{grcomplexsection}), we get an injection from Teichmüller space into $\hat{\mc{T}}_{\g_\R}$.  Further, we get the following theorem:

\begin{thm}\label{splitmodulispace}
There is a map $$\hat{\mc{T}}_{\g_\R} \rightarrow \hat{\mc{T}}_{\g}$$
which is an isomorphism for the split real form $\g_{split}$.
\end{thm}
\begin{proof}
Equation \eqref{grgc} gives the map between the corresponding Hilbert schemes. Further, a higher diffeomorphism of type $\g_\R$ is always a higher diffeomorphism of type $\g$. 

Consider now the split real form. By Theorem \ref{split-complex}, we have $\Hilb^{reg}_0(\g_{split}) \cong \Hilb^{reg}_0(\g)$. This implies that a $\g_{split}$-complex structure is the same as a $\g$-complex structure.

Finally, we prove that a higher diffeomorphism of type $\g_ {split}$ is the same as a higher diffeomorphism of type $\g$.
Again by Theorem \ref{split-complex}, the possible eigenvalues of a pair $[(A,B)]\in \Hilb^{reg}(\g_{split})$ are the same as for a pair in $\Hilb^{reg}(\g)$, since both are in bijection. So the extra symmetry which has to be preserved by a higher diffeomorphism is the same. The equivalence between the two notions of higher diffeomorphisms, and thus of the two moduli spaces follows.
\end{proof}
It would be very interesting to compute the moduli space for a non-split real form.

In \cite{HKR-section}, the authors construct an analogue to the Hitchin fibration between the moduli space of $G_\R$-Higgs bundles and the Hitchin base, using invariant polynomials. They also construct a section to this fibration, called the \textit{Hitchin-Kostant-Rallis section}. In the split real case, this simply gives the Hitchin section.
 
We close by enlarging Conjecture \ref{conjmaj} about the link between our moduli space and Hitchin's component to the following:
\begin{conj}
Our moduli space $\hat{\mc{T}}_{\g_\R}$ is canonically homeomorphic to the image of the Hitchin-Kostant-Rallis section.
\end{conj}
Note that the image of the HKR-section through the non-abelian Hodge correspondence is not in general a component in the character variety $\Rep(\pi_1(\S), G_\R)$.

To attack this conjecture, one should either try to generalize the techniques developed in \cite{Thomas} (where several steps for the case $\mf{sl}_n(\C)$ has been proven), or to use a representation $\rho:\g \rightarrow \mf{sl}_m(\R)$ and Theorem \ref{splitmodulispace} to recast the problem in the realm of $\mf{sl}_m(\C)$.

\appendix

\section{Punctual Hilbert schemes revisited}\label{appendix:A}

In this appendix, we review the punctual Hilbert scheme of the plane with its various viewpoints. Main references are Nakajima's book \cite{Nakajima} and Haiman's paper \cite{Haiman}.

\subsection{Definition}

To start, consider $n$ points in the plane $\mathbb{C}^2$ as an algebraic variety, i.e. defined by some ideal $I$ in $\mathbb{C}[x,y]$. Its function space $\mathbb{C}[x,y]/I$ is of dimension $n$, since a function on $n$ points is defined by its $n$ values. So the ideal $I$ is of codimension $n$. 
The space of all such ideals, or in more algebraic language, the space of all zero-subschemes of the plane of given length, is the punctual Hilbert scheme:

\begin{definition}
The \textbf{punctual Hilbert scheme} $\Hilb^n(\mathbb{C}^2)$ of length $n$ of the plane is the set of ideals of $\mathbb{C}\left[x,y\right]$ of codimension $n$: 
$$\Hilb^n(\mathbb{C}^2)=\{I \text{ ideal of } \mathbb{C}\left[x,y\right] \mid \dim(\mathbb{C}\left[x,y\right]/I)=n \}.$$  
\noindent The subspace of $\Hilb^n(\mathbb{C}^2)$ consisting of all ideals supported at 0, i.e. whose associated algebraic variety is $(0,0)$, is called the \textbf{zero-fiber} of the punctual Hilbert scheme and is denoted by $\Hilb^n_0(\mathbb{C}^2)$.
\end{definition}

A theorem of Grothendieck and Fogarty asserts that $\Hilb^n(\mathbb{C}^2)$ is a smooth and irreducible variety of dimension $2n$ (see \cite{Fogarty}). The zero-fiber $\Hilb^n_0(\mathbb{C}^2)$ is an irreducible variety of dimension $n-1$, but it is in general not smooth. 

A generic element of $\Hilb^n(\C^2)$, geometrically given by $n$ distinct points, is given by
$$I=\left\langle x^n+t_1x^{n-1}+\cdots+t_n,-y+\mu_1+\mu_2x+...+\mu_nx^{n-1}\right\rangle.$$
The second term can be seen as the Lagrange interpolation polynomial of the $n$ points.

A generic element of the zero-fiber is given by
$$I=\left\langle x^n,-y+\mu_2x+...+\mu_nx^{n-1}\right\rangle.$$

\subsection{Resolution of singularities}\label{resofsing}

Given an ideal $I$ of codimension $n$, we can associate its support, the algebraic variety defined by $I$, which is a collection of $n$ points (counted with multiplicity). The order of the points does not matter, so there is a map, called the \textbf{Chow map}, from $\Hilb^n(\mathbb{C}^2)$ to $\Sym^n(\mathbb{C}^2) := (\mathbb{C}^2)^n/\mathcal{S}_n$, the configuration space of $n$ points ($\mathcal{S}_n$ denotes the symmetric group). A theorem of Fogarty asserts that the punctual Hilbert scheme is a \textit{minimal resolution of the configuration space}.

In order to get a feeling for a general Lie algebra, notice that $n$ points of $\mathbb{C}^2$ is the same as two points in the Cartan $\mathfrak{h}$ of $\mathfrak{gl}_n$, and that the symmetric group is the Weyl group $W$ of $\mathfrak{gl}_n$. So the configuration space equals $\h^2/W$ for $\g=\mf{gl}_n$.

\subsection{Matrix viewpoint}

To an ideal $I$ of codimension $n$, we can associate two matrices: the multiplication operators $M_x$ and $M_y$, acting on the quotient $\C[x,y]/I$ by multiplication by $x$ and $y$ respectively. To be more precise, we can associate a conjugacy class of the pair: $[(M_x,M_y)]$.

The two matrices $M_x$ and $M_y$ commute and they admit a cyclic vector, the image of $1 \in \C[x,y]$ in the quotient (i.e. 1 under the action of both $M_x$ and $M_y$ generate the whole quotient).

\begin{prop}\label{bijhilbert}
There is a bijection between the Hilbert scheme and conjugacy classes of certain commuting matrices: $$\Hilb^n(\mathbb{C}^2) \cong \{(A,B) \in \mf{gl}_n^2 \mid [A,B]=0, (A,B) \text{ admits a cyclic vector}\} / GL_n$$
\end{prop}

The inverse construction goes as follows: to a conjugacy class $[(A,B)]$, associate the ideal $I=\{P \in \C[x,y] \mid P(A,B)=0\}$, which is well-defined and of codimension $n$ (using the fact that $(A,B)$ admits a cyclic vector). For more details see \cite{Nakajima}.

It is this bijection which we use in the main text to generalize the punctual Hilbert scheme.
Notice that the \textit{zero-fiber of the Hilbert scheme corresponds to nilpotent commuting matrices}.

\subsection{Reduced Hilbert scheme}

We wish to define a subspace of $\Hilb^n(\C^2)$ corresponding to matrices in $\mf{sl}_n$ in the matrix viewpoint. 
A generic point should be a pair of points in the Cartan $\h$ of $\mf{sl}_n$ modulo order. This corresponds to $n$ points in the plane with barycenter 0.

\begin{definition}
The \textbf{reduced Hilbert scheme} $\Hilb^n_{red}(\C^2)$ is the space of all elements of $\Hilb^n(\C^2)$ whose image under the Chow map ($n$ points with multiplicity modulo order) has barycenter 0.
\end{definition}

With this definition, we get
\begin{prop}
$$\Hilb^n_{red}(\mathbb{C}^2) \cong \{(A,B) \in \mf{sl}_n^2 \mid [A,B]=0, (A,B) \text{ admits a cyclic vector}\} / SL_n.$$
\end{prop}

Finally, it can be proven that the reduced Hilbert scheme is symplectic and that the zero-fiber $\Hilb^n_0(\C^2)$ is a Lagrangian subspace of $\Hilb^n_{red}(\C^2)$.

\section{Regular elements in semisimple Lie algebras}\label{appendix:B}

In this appendix, we gather all properties we need in the main text of regular elements in semisimple Lie algebras and we give precise references for these results. The main references are the books of Collingwood and McGovern \cite{Coll}, Steinberg \cite{Steinberg} and Humphreys \cite{Hum}, as well as the papers \cite{Kost} and \cite{Kost2} by Kostant.

\begin{definition}
An element $x \in \g$ is called \textbf{regular} if the dimension of its centralizer $Z(x)$ is equal to the rank of the Lie algebra $\rk(\g)$. A regular nilpotent element is called \textbf{principal nilpotent}.
\end{definition}

\begin{Remark} 
Notice that in older literature, regular elements are defined in another way, using the characteristic polynomial of the adjoint map. The ``old'' notion includes only semisimple regular (in the sens above) elements. 
\end{Remark}

The condition that the dimension of the centralizer has to be equal to the rank, does not come from nowhere: in fact it is the minimal possible dimension.

\begin{prop}
For any $x \in \g$, we have $\dim Z(x) \geq \rk(\g)$.
\end{prop}
See for example Lemma 2.1.15. in \cite{Coll}.

For the Lie algebras $\mf{gl}_n$ and $\mf{sl}_n$, we have the following characterization of regular elements from Steinberg \cite{Steinberg}, Proposition 2 in Section 3.5:

\begin{prop}\label{regularsln}
For $\g=\mf{gl}_n$ or $\mf{sl}_n$ and $x\in \g$, we have the following equivalence:
$$x \text{ is regular} \Leftrightarrow \mu_x = \chi_x \Leftrightarrow x \text{ admits a cyclic vector}$$ where $\mu_x$ and $\chi_x$ denote respectively the minimal and the characteristic polynomial of $x$, seen as a matrix. 
\end{prop} 

Let us turn to the study of regular elements which are nilpotent. 
%Define $\mathcal{N}$ to be the set of nilpotent elements in $\g$. 

\begin{thm}
There is a unique open dense orbit in the nilpotent variety consisting of principal nilpotent elements.
\end{thm}
The original proof is due to Kostant, see Corollary 5.5. in \cite{Kost}.
See also Theorem 4.1.6. in \cite{Coll}.

There is a useful characterization of principal nilpotent elements in coordinates. For this, fix a root system $R$, and a direction giving the positive roots $R_+$. Denote by $(e_\alpha)_{\alpha \in R}$ a basis of $\g$ given by root vectors. Denote by $\mf{n}_+$ the positive nilpotent elements generated by $(e_\alpha)_{\alpha \in R_+}$ (for $\mf{sl}_n$, we get upper triangular matrices).
\begin{prop}\label{prinnilp}
Let $A \in \mf{n}_+$. Then $A = \sum_{\alpha \in R_+}A_\alpha e_\alpha$ is principal nilpotent iff $A_\alpha\neq 0$ for all simple roots $\alpha$.
\end{prop}
This proposition can be found in \cite{Kost}, Theorem 5.3.

For a principal nilpotent element $f$, its centralizer $Z(f)$ has properties quite analogous to a Cartan, the centralizer of a regular semisimple element:

\begin{thm}\label{thmKost}
For $f$ a principal nilpotent element, its centralizer $Z(f)$ is abelian and nilpotent.
\end{thm}
Kostant proves even more, using a limit argument: for any element $x\in \g$, there is an abelian subalgebra of $Z(x)$ of dimension $\rk \g$, see \cite{Kost}, theorem 5.7.
The nilpotency of $Z(f)$ can be found in \cite{Steinberg}, corollary in Section 3.7. The more precise structure of $Z(x)$ for any nilpotent $x$ is described in \cite{Coll}, Section 3.4.

A principal nilpotent element permits to give a preferred representative of a conjugacy class of regular elements. Given $f$ principal nilpotent, denote by $e$ the other nilpotent element in a principal $\mf{sl}_2$-triple constructed from $f$ (see Kostant \cite{Kost}). Then we get
\begin{prop}
Any regular orbit intersects $f+Z(e)$ in a unique point. So we have $\g^{reg}/G \cong f+Z(e)$.
\end{prop}
This follows from Lemma 10 of \cite{Kost2}. The set $f+Z(e)$ is called a \textit{principal slice} of $\g$ (also \textit{Kostant section}).

We are now going to ``double'' the previous setting. Define the commuting variety to be $\Comm(\g):=\{(A,B)\in \g^2 \mid [A,B]=0\}$. 

\begin{thm}[Richardson]\label{Richardson}
The set of commuting semisimple elements is dense in the commuting variety $\Comm(\g)$.
\end{thm}
See the paper of Richardson \cite{Richardson} for a proof.
As a consequence, $\Comm(\g)$ is an irreducible variety, but highly singular.

With this, we can explore the minimal dimension of a centralizer of a commuting pair:
\begin{prop}\label{doublecomm}
For $(A,B) \in \Comm(\g)$, we have $\dim Z(A,B) \geq \rk \g.$
\end{prop}
\begin{proof}
Consider the set $M$ of elements with centralizer of minimal dimension. Since $$M=\{(A,B) \in \Comm(\g) \mid \rk(ad_A, ad_B) \text{ maximal}\}$$ we see that $M$ is Zariski-open. By the theorem of Richardson it intersects the space of semisimple pairs for which the common centralizer is a Cartan $\h$, so of dimension $\rk \g$.
\end{proof}

\section{Haiman coordinates}\label{haimancoords}
We already mentioned a uniform way to get coordinates in the chart of $\Hilb^n(\C^2)$ associated to a Young diagram $D$. These are described by Haiman in his paper \cite{Haiman}. We describe them here for completeness and to give a detail of its symplectic structure which seems to be new.

The construction of Haiman's coordinates goes as follows: For each box $B_x \in D$ consider the rightmost box $B_r \in D$ in the same row as $B_x$ and the bottommost box $B_b \in D$ in the same column as $B_x$ (see Figure \ref{Haimancoo}). The box $B_{r+1}$ to the right of $B_r$ is not in $D$, so gives a linear combination of boxes in $D$. Denote by $b_{x,r}$ the coefficient of $B_b$ in this linear combination. Similarly, denote by $b_{x,b}$ the coefficient of $B_r$ in the linear combination associated to the box $B_{b+1}$ at the bottom of $B_b$. Haiman shows that the set $\{b_{x,r}, b_{x,b}\}_{x\in D}$ is a coordinate system. 
\begin{figure}[h]
\centering
\includegraphics[height=2cm]{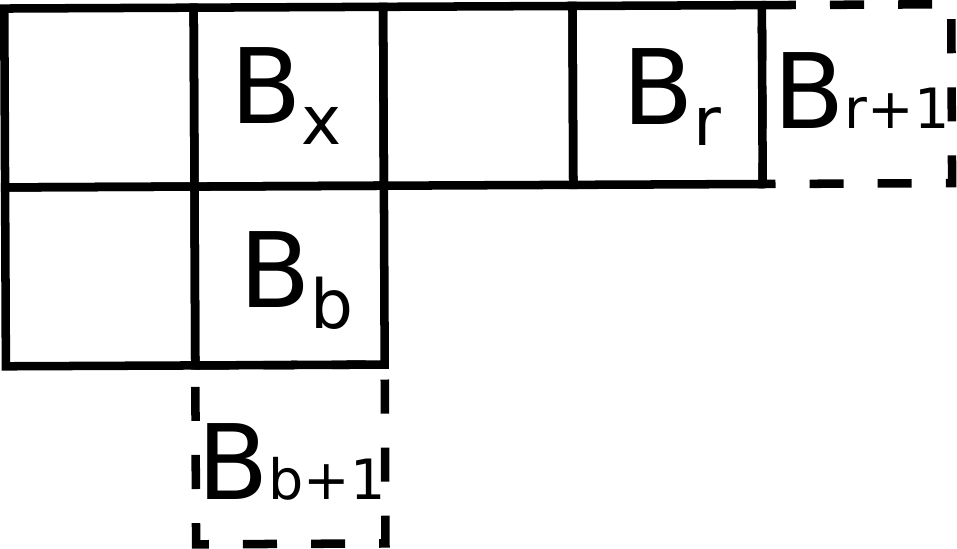}

\caption{Haiman's coordinates}\label{Haimancoo}
\end{figure}

One can try to write the symplectic structure in the Haiman coordinates. In general, this gives a quite complicated expression. In one special case, the situation is easy:

\begin{prop}
If the Young diagram is a rectangle, then Haiman's coordinates $\{b_{x,r}, b_{x,b}\}_{x\in D}$ are canonical coordinates with respect to the symplectic structure of the punctual Hilbert scheme.
\end{prop}
The idea of the proof is to compute the symplectic form in the basis adapted to the Young diagram.
\begin{proof}
The symplectic structure of the punctual Hilbert scheme comes from the canonical symplectic structure of $\C^{2n}$ given by $\omega=\sum_i dx_i\wedge dy_i.$
Consider the multiplication operators $M_x$ and $M_y$ in the quotient $\C[x,y]/I$ where $I$ is an element in the Hilbert scheme (idealic viewpoint). Diagonalizing these operators give diagonal matrices with entries $(x_1,...,x_n)$ and $(y_1,...,y_n)$. Hence we can express the symplectic structure by $$\omega = \sum_i dx_i\wedge dy_i = \tr dM_x \wedge dM_y.$$

Changing to the base adapted to the Young diagram $D$ (basis generated by monomials $x^iy^j$ where $(i,j) \in D$), the matrix $M_x$ becomes a matrix $N_x$ with entries 1 on the line under the diagonal, apart from some columns where the linear combination associated to some $B_{r+1}$ is written. Similarly, the matrix $M_y$ becomes a matrix $N_y$ where the only columns which are non-constant are the last ones where the linear combination associated to the $B_{b+1}$ are written. 

If we denote by $T$ the transition matrix, we get $d(TM_xT^{-1}) = d(T) M_xT^{-1} + Td(M_x)T^{-1}+TM_xd(T^{-1})$. A lengthy but straight forward computation, using the cyclicity of the trace and the fact that $M_x$ and $M_y$ commute, shows that $$\tr M_x \wedge M_y = \tr d(TM_xT^{-1}) \wedge d(TM_yT^{-1}) = \tr N_x \wedge N_y.$$

Let us now use the fact that the Young diagram is a rectangle, say with $k$ rows and $l$ columns. By definition of the Haiman coordinates, we can compute where they appear in $N_x$ and $N_y$. For $\alpha \in [1,...,k]$ and $\beta \in [1,...,l]$, we get 
$$(N_x)_{l(k-1)+\beta,\alpha l} = b_{(\alpha, \beta),r} \text{ and } (N_y)_{\alpha l,l(k-1)+\beta} = b_{(\alpha, \beta),b}.$$

Finally, since the only non-zero rows of $dN_x$ are those in position $\alpha l$ and the only non-zero columns of $dN_y$ are those in position $l(k-1)+\beta$, we see that $\tr dN_x \wedge dN_y = \sum_{x\in D} db_{x,r}\wedge db_{x,b}$. Therefore we conclude:
$$\omega = \tr dM_x \wedge dM_y = \tr dN_x \wedge dN_y = \sum_{x\in D} db_{x,r}\wedge db_{x,b}.$$
\end{proof}

\end{document}